\documentclass[preprint,11pt]{elsarticle}
\usepackage{amssymb}
\usepackage{graphicx}
\usepackage{epsfig}
\usepackage{color}
\usepackage{subfigure}
\usepackage{subfigmat}
\usepackage{amsthm}
\usepackage{latexsym}
\usepackage{amsmath}
\usepackage{rotating}
\usepackage{amsbsy}
\usepackage{mathrsfs}
\usepackage{lineno}
\usepackage{hyperref}
\usepackage{booktabs} 
\usepackage{float}     
\usepackage{caption}
\usepackage{subcaption}
\newtheorem{thm}{Theorem}[section]
\newtheorem{example}{Example}[section]
\newtheorem{prop}{Proposition}[section]

\newtheorem{rem}{Remark}[section]

\newtheorem*{LemmaA.1}{Lemma A.1}
\newtheorem*{LemmaA.2}{Lemma A.2}
\numberwithin{equation}{section}

\journal{}

\begin{document}

\begin{frontmatter}

\title{$p$ and $rp$-Spectral Element Methods for Two-dimensional Elliptic Boundary Layer Problems}
\author[label1]{Sonia}
\ead{sonia2400861@st.jmi.ac.in}
\author[label1]{Akhlaq~Husain\corref{cor1}}
\ead{ahusain10@jmi.ac.in}
\author[label1]{Arshad~Khan}
\ead{akhan2@jmi.ac.in}
\address[label1]{Department of Mathematics, Faculty of Sciences, Jamia Millia Islamia, New Delhi-110025, Delhi, India}
\cortext[cor1]{Corresponding author}

\begin{abstract}
We propose $p$ and $rp$-spectral element methods for elliptic boundary layer problems on two dimensional rectangular domains. To resolve boundary layers, we use $p$-version with a fixed mesh and an $rp$-version with a boundary layer mesh consisting of thin needle-like elements near the boundary layer and coarse elements away from the layer. Stability estimates are derived using non-conforming spectral element functions. The numerical scheme for both the methods is based on minimising a least-squares functional in appropriate Sobolev norms. We construct robust preconditioners to manage the condition number of the normal equations derived from the least-squares formulation. The method is able to approximate boundary layers at a rate $O\left(\frac{\log W}{W^2}\right)$ for the $p$-version and at the rate $O({\epsilon}\alpha^{2W})$, uniformly in $\epsilon$, for the $rp$-version, where $0<\epsilon\leq1$ is the boundary layer parameter, $\alpha<1$ is a constant and $W$ denotes the degree of the approximating polynomial. Numerical results are provided for model elliptic boundary layer problems for a range of boundary layer parameters. Simulation results demonstrate the efficiency and robustness of the method in capturing boundary layers on rectangular domains.
\end{abstract}


\end{frontmatter}

\section{Introduction}\label{sec1}

Elliptic boundary layer problems arise from a class of singular perturbation problems (see~\cite{Eckhaus1972, M2} and the references therein). These problems involve an elliptic differential equation in which the highest-order derivative is scaled by a small positive parameter. As a result, sharp gradients develop in confined regions of the domain, referred to as boundary layers. This causes the solution to develop sharp gradients in narrow regions, known as boundary layers, which are supported in a thin neighbourhood along the boundary of the domain. These problems primarily arise in the modelling of several physical phenomena, particularly in semiconductor device modelling, plate and shell theories in structural mechanics, and in the Navier–Stokes equations under low-viscosity flow, where small parameters lead to rapid variations in the solution near the boundary. These features pose significant challenges for standard numerical methods~\cite{roos2008,DMS}. In particular, conventional finite element and spectral methods often fail to resolve these layers accurately without proper mesh refinement near the boundary of the domain~\cite{Deville2002}.

In some rare scenarios, the numerical methods may not even converge. The numerical approximation of these problems requires careful construction in order to accurately capture the boundary layer behaviour. Designing a numerical method for these problems requires both precision and computational efficiency, which remains a significant challenge.

For the effective resolution of boundary layers, a large amount of work has been carried out using various existing methods, such as finite difference or lower-order finite element methods, in combination with several mesh refinement techniques based on the $h$, $p$, and $hp$-versions (such as the Shishkin~\cite{SHISH} or Bakhvalov~\cite{BAKH} meshes; see~\cite{ROO1} for a detailed overview), or through the use of high-order $p$ and $hp$-versions on appropriately designed meshes (see~\cite{SCHW1,SCHW2,XENO1,XENO2} and the references therein). In addition, $rp$-version methods, where the number of elements is fixed and the element sizes vary with the polynomial degree, provide an effective framework for resolving boundary layers~\cite{SCHW1}.

To develop a robust numerical technique for effectively resolving boundary layers, the spectral element method has emerged as a highly effective approach due to its high-order accuracy~\cite{SSX1}. It has been extensively used to solve partial differential equations over the past two to three decades (refer to~\cite{CHQ1,KS,ST} and the related references). The combination of the high accuracy of spectral methods and the flexibility of finite element methods makes least-squares spectral element methods an especially effective technique.
These methods have been widely applied to compute approximate solutions of elliptic problems on both smooth and non-smooth domains~\cite{DHMU1,DHMU2,DHMU3,DKU,DTK2,AKCSU,KR,KR1}, as well as to various physical problems, such as the Navier–Stokes equations~\cite{PR,PG}. Furthermore, they provide considerable mathematical and computational benefits, especially in the structural design and algorithmic execution of the least-squares framework (see~\cite{BOCH,DHMU1,DHMU2,DHMU3,DKU,DTK2,AKCSU,KR}). In the spectral element approach, the solution is approximated using high-order spectral elements.

In this article, we propose $p$ and $rp$-least-squares spectral element methods for elliptic boundary layer problems on two-dimensional rectangular domains with Dirichlet boundary conditions. The method is based on minimizing the sum of the squares of residuals in the partial differential equation, the sum of the squared norms of residuals in the boundary conditions in fractional Sobolev norms, and the sum of the squared norm of jumps in the function and its derivatives across the interface in relevant fractional Sobolev norms to enforce the continuity along the inter-element boundary.

To compute the solution, we solve normal equations using a preconditioned conjugate gradient method. For computing the residuals, the integrals are efficiently approximated using Gauss–quadrature rules. If the input data is analytic and $W$ denotes an upper bound on the degree of the approximating polynomials, then for the $p$-version, the proposed method is able to resolve boundary layers at a robust rate ${O}\left( \frac{\log W}{W^2} \right)$, and for the $rp$-version the convergence rate is uniform in $\epsilon$, equalling ${O}(\epsilon^{\alpha} W)$ on a boundary layer mesh, where $0 < \epsilon \leq 1$ is the boundary layer parameter and $\alpha < 1$ is a constant.

The paper is organized as follows: In Section~2, presents the elliptic boundary layer problem along with essential notations and function spaces. Section~3 presents domain discretization and stability estimates. Design of numerical scheme with the help of stability estimates are presented in Section~4. To deal with normal equations arising from the numerical scheme, preconditioning techniques are discussed in Section~5. Error estimates are calculated in Section~6. Numerical results are detailed in Section 7, and final conclusions are provided in Section 8.

\section{Elliptic boundary layer problem}\label{sec2}

\subsection{Notations and preliminaries}
We denote an arbitrary element of $\mathbb{R}^2$ by $\mathbf{x}=(x_{1},x_{2})$ or $\mathbf{x}=(x,y)$, $|\mathbf{x}|=\sqrt{x^2+y^2}$.
Let $\Omega\subset\mathbb{R}^2$ be open and bounded, with boundary \(\partial\Omega\).  
The standard Sobolev space of order $m\in\mathbb{N}$, denoted by $H^{m}(\Omega)$, is equipped with the norm
\begin{align}\label{eq2.1}
\|u\|_{m,\Omega}^{2}=\sum_{|\alpha|\leq m}\|\partial^\alpha u\|^2_{0,\Omega}=\int_{\Omega} \sum_{|\alpha|\leq m}|\partial^{\alpha} u|^{2}d{\mathbf{x}},
\end{align}
where $\alpha=(\alpha_1,\alpha_2)$, $\alpha_i\geq 0$ are integers for $i=1,2$, $|\alpha|=\alpha_1+\alpha_2$. The derivative
\[\partial^{\alpha}u=\frac{\partial^{|\alpha|}u}{\partial x_1^{\alpha_1}\partial x_2^{\alpha_2}}\]
is understood in the distributional (weak) sense. \\ As usual $H^0(\Omega)=L^2(\Omega)$.
The subspace of functions with zero trace on $\partial\Omega$ is 
\begin{align*}
H^m_{0}(\Omega)=\{u\in H^{m}(\Omega): u=0 \:\mbox{on}\: \partial\Omega\},
\end{align*}
and the space of functions with non-zero trace on the boundary is
\begin{align*}
H^m_{D}(\Omega) =\{u\in H^{m}(\Omega): u=u|_{\partial\Omega} \:\mbox{on}\: \partial\Omega\},
\end{align*}
where $u|_{\partial\Omega}$ denotes the trace of $u$ on $\partial\Omega$.\\
A seminorm on $H^{m}(\Omega)$ is given by
\begin{align}\label{eq2.2}
|u|_{m,\Omega}^{2}=\sum_{|\alpha|=m}\|\partial^\alpha u\|^2_{0,\Omega}=\int_{\Omega} \sum_{|\alpha|=m}|\partial^{\alpha} u|^{2}d{\mathbf{x}}.
\end{align}
Further, the fractional Sobolev norm of order $\sigma$ for $0 < \sigma < 1$ as
\begin{align}\label{eq2.3}
\|u\|_{\sigma,I}^{2} = \|u\|_{0,I}^{2} + \int_I\int_I\frac{\left(u(x_{1}) - u(x_{2})\right)^2}{\left|x_{1} - x_{2}\right|^{1 + 2\sigma}}\,dx_{1} \,dx_{2},
\end{align}
where $I$ denotes an interval in $\mathbb{R}$. Moreover,
\begin{align}\label{eq2.4}
\|u\|_{1+\sigma,I}^2 = \|u\|_{0,I}^2+\left\|\:\frac{\partial u}{\partial x}\:\right\|_{\sigma,I}^2.
\end{align}
For the parameter $\epsilon > 0 $, we define the parameter-weighted norm
\begin{align}\label{eq2.5}
\|u\|^2_{m,\epsilon,\Omega}=\sum_{|\alpha|\leq m}\epsilon^{2|\alpha|}\|\partial^\alpha u\|^2_{0,\Omega}=\int_{\Omega}\sum_{|\alpha|\leq m}\epsilon^{2|\alpha|}|\partial^{\alpha}u|^{2}d{\mathbf{x}},
\end{align}
The associated function space is 
\begin{align*}
H^m_{\epsilon}(\Omega) = \left\{ u \in H^m(\Omega) \;:\; \|u\|_{m,\epsilon,\Omega} < \infty \right\}.
\end{align*}
If \(\varphi \in H^{m-\frac12}(\partial\Omega)\) is a prescribed boundary function, we define the space
\begin{align*}
H^m_{\epsilon,D}(\Omega) = \left\{ u \in H^m_{\epsilon}(\Omega) \;:\; u = \varphi \ \text{on} \ \partial\Omega \right\}.
\end{align*}

\subsection{The model problem}
Let $\Omega = (-1,1)^2 \subset \mathbb{R}^2$ be the rectangular domain shown in Figure~\ref{fig1}, with boundary $\Gamma = \partial\Omega$.
\begin{figure}[ht]
\centering
{\includegraphics[width=0.26\textwidth]{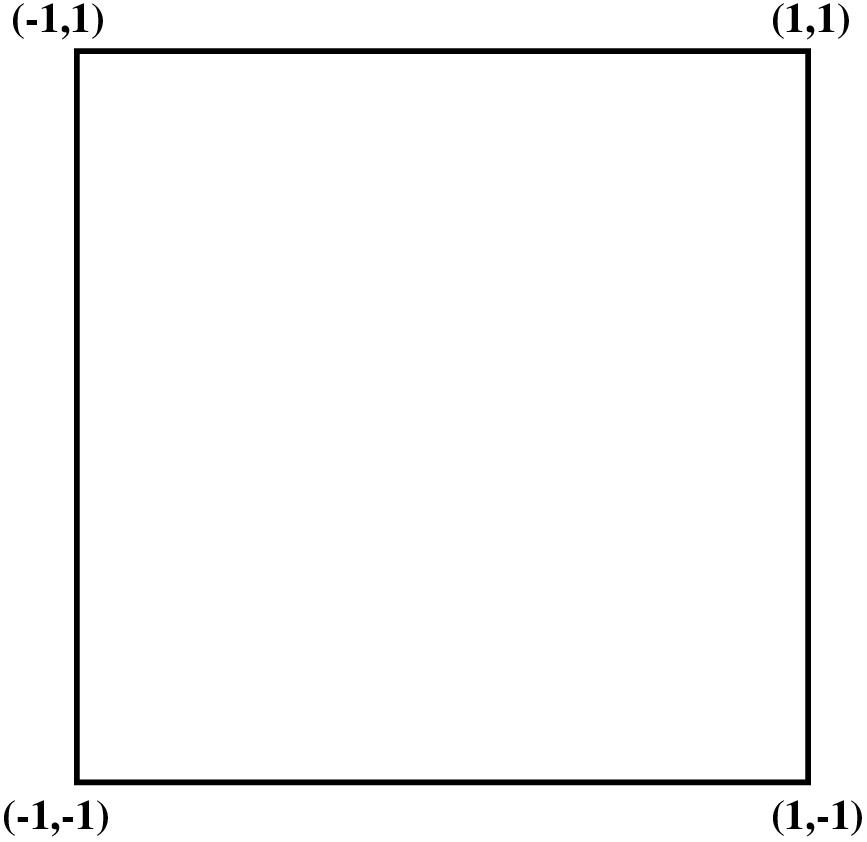}}
\caption{The domain $\Omega=(-1,1)^2$}
\label{fig1}
\end{figure}
We consider the elliptic boundary layer problem
\begin{align}\label{eq2.6}
-\epsilon^2\Delta w+aw=f\quad\mbox{in}\quad\Omega,
\end{align}
with Dirichlet boundary conditions
\begin{align}\label{eq2.7}
w=g\quad\mbox{on}\quad\Gamma.
\end{align}
Here, $f\in L^2(\Omega), g\in H^{\frac{3}{2}}(\Gamma)$, $a > 0 $ is a real constant and $0<\epsilon\leq 1$ is a small parameter.
The problem (\ref{eq2.6})–(\ref{eq2.7}) is well posed. Because the domain $\Omega=(-1,1)^2$ has corners, its boundary $\Gamma$ is
only Lipschitz continuous. It is well known~\cite{M2,MS1,XENO1,XENO2} that the solution $w$ exhibits regular boundary layers along
the four smooth edges of the domain, as well as corner layers at the vertices. Along any single smooth edge (away from the corners),
the regular boundary layer takes the generic one dimensional-like form
\begin{align}\label{eq2.8}
w_{bl}(\rho,\theta)=\phi(\theta)\exp\left(\frac{-\sqrt{a}\rho}{\epsilon}\right),
\end{align}
where $\rho$ is the normal distance to the edge, $\theta$ is the local coordinate along that edge, and $\phi(\theta)$ is a smooth function
defined strictly on the open interval of the edge. For instance, in Cartesian coordinates along the bottom edge ($y=-1$), this regular edge
layer takes the form
\begin{align}\label{eq2.9}
w_{bl}(x,y)=\psi(x)w_\epsilon(y)=\psi(x)e^{-\frac{(y+1)}{\epsilon}},
\end{align}
where $\psi(x)$ is smooth (analytic) for $x \in (-1,1)$ and $w_\epsilon(y) = e^{-\frac{(y+1)}{\epsilon}}$ represents the one-dimensional
exponential decay. As our primary goal is to derive parameter-uniform error estimates for capturing these steep gradients, our theoretical
analysis will focus on resolving this dominant one-dimensional exponential decay in the normal direction, modeled by (\ref{eq2.8})–(\ref{eq2.9}).

Therefore, in order to design an efficient numerical scheme we need to approximate the solution efficiently in the $\rho$
direction in the space of spectral element functions. Using techniques and results for the 1D problem~\cite{HUS2,HUS3} we
will construct spectral element spaces which are the tensor product of 1D spectral element spaces to approximate functions
of the form (\ref{eq2.8})$-$(\ref{eq2.9}) similar to the FEM analysis in~\cite{M2,MS1,XENO1,XENO2}. Moreover, we shall
consider higher order $p$ and $rp$ methods to resolve the boundary layers at a sufficiently fast rate.

The following regularity estimate~\cite{XENO1,XENO2} holds.
\begin{thm}[Regularity]\label{thm1}
If $f\in L^{2}(\Omega)$ and $g\in H^{3/2}(\Gamma)$, then there is a unique solution $w\in H^2_{\epsilon,D}(\Omega)$
to the problem (\ref{eq2.6})-(\ref{eq2.7}) and
\begin{equation}\label{eq2.10}
\|w\|_{2,\epsilon,\Omega} \leq C(\epsilon)\left(\|f\|_{0,\Omega}+\|g\|_{\frac{3}{2},\Gamma}\right),
\end{equation}
where the constant $C(\epsilon)$ depends on $\epsilon$.
\end{thm}

\begin{rem}
Note that the regularity estimate in (\ref{eq2.10}) are non-uniform in $\epsilon$ since the constant $C(\epsilon)$ depends on $\epsilon$.
\end{rem}

The solution $w(x)$ to the problem (\ref{eq2.6})$-$(\ref{eq2.7}) can be decomposed into a regular (smooth) part, a boundary layer part and a smooth remainder using a set of boundary fitted coordinates~\cite{XENO1,XENO2}.

\section{Domain Discretization and stability estimate}\label{sec3}

\subsection{Discretization and spectral element representation}
We decompose the domain $\Omega$ into a collection of quadrilaterals (say) $\Omega_{1},\Omega_{2},\ldots,\Omega_{N}$ (rectangles or squares) using a uniform (or non-uniform) refinement in each coordinate direction ensuring the sub-domain partitions
coincide with the the inter element boundaries (see Figure~\ref{fig2}). Since a triangle can be partitioned into three quadrilaterals by connecting the centroid of the triangle to the midpoints of the sides therefore, we may conveniently choose quadrilateral elements throughout the domain. Let each element uses an approximating polynomial of degree at most $W$. We shall define the space of
spectral element functions on $\Omega$ to be the space consisting of  functions formed by tensor-products of polynomials of degree $W$ in
each direction on $\Omega_l$, $1\leq l\leq N$. 
\begin{figure}[ht!]
\centering
\includegraphics[width=0.30\textwidth]{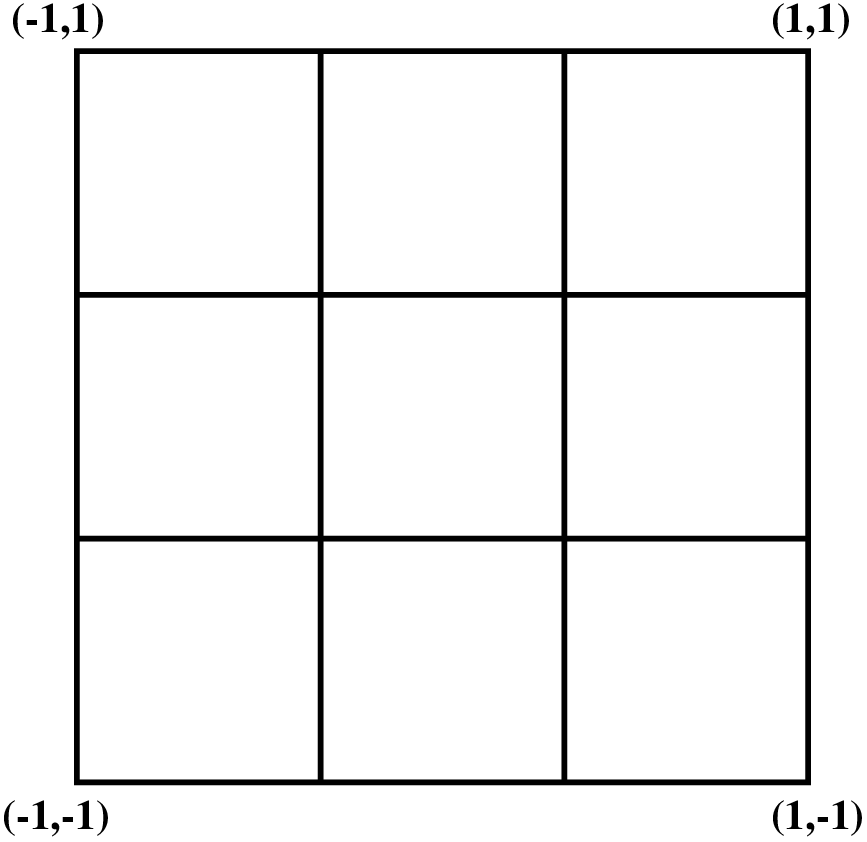}
\caption{Domain $\Omega$ discretized into elements $\Omega_l$ for $1 \leq l \leq N$.}
\label{fig2}
 \end{figure}
Let $S=(-1,1)^2$, denote the bi-unit square. Then there is an analytic map $M_{l}:S\rightarrow\Omega_{l}$
having an analytic inverse. The map $M_{l}$ is defined by (see~\cite{BG4})
\begin{align}\label{eq3.1}
x_1=X_{l}(\lambda),\;\;\mbox{and}\;\;x_2=Y_{l}(\lambda),\:\: \lambda=(\xi,\eta)\in S,\:\: l=1,2,\cdots,N.
\end{align}
We now introduce the spectral element functions $\{\hat{u}_{l}\}$ on the reference element $S$, defined as tensor products of degree $W$ polynomials in $\xi$ and $\eta$ variables, given by:

\begin{align}\label{eq3.2}
\hat{u}_{l}(\lambda)=\hat{u}_{l}(\xi,\eta)=\sum_{r=0}^{W}\sum_{s=0}^{W}\alpha_{r,s} \xi^{r}\eta^{s}.
\end{align}
Then, the spectral element functions $\{u_{l}\}$ on $\Omega_{l}$ are given by
\begin{align}\label{eq3.3}
u_{l}(x)=u_{l}(x_1,x_2)= \hat{u}_{l}(M_{l}^{-1}).
\end{align}

\subsection{Stability estimate}
Let $Lu=-\epsilon^2\Delta u+au$ and let $J^{l}(\lambda)$ be the Jacobian of the map $M_{l}:S\rightarrow\Omega_{l}$ for $l=1,2,\cdots,N$. Then,
\begin{align}\label{eq3.4}
\int_{\Omega_{l}} |Lu_{l}|^{2} d{\mathbf{x}}=\int_{S}|L\hat{u}_{l}|^{2}|J_{l}|\, d\xi d\eta.
\end{align}
where,
\[L\hat{u}_{l}=-\epsilon^2\left(\hat{u}_{\xi\xi}+\hat{u}_{\eta\eta}\right)+\hat{a}\hat{u}, \:\: \hat{a}={a}(\xi,\eta).\]
Define ${L}_{l}=L\sqrt{J_l}$. Then,
\begin{align}\label{eq3.5}
\int_{\Omega_{l}} |Lu_{l}|^{2} d{\mathbf{x}}=\int_{S}|L_l\hat{u}_{l}|^{2} \,d\xi d\eta.
\end{align}
First, we introduce the requisite notation for defining the functional associated with the stability estimate. Let $\Omega_{l}$ be an element in the partition of $\Omega$ with $\{\Gamma_{l}^j\}_{1\leq j\leq 4}$ be its edges. Let $\Omega_m$ and $\Omega_n$, be
two adjacent elements in the sub-division of the domain $\Omega$ and $\Gamma^{j}_{m}$ be an edge of $\Omega_m$ and
$\Gamma^{k}_n$ be an edge of $\Omega_n$ such that $\Gamma^{j}_m=\Gamma^{k}_n$, i.e. $\Gamma^{j}_m$ is an edge common to
$\Omega_m$ and $\Omega_n$.

Let $\left.{[u]}\right|_{\Gamma^{j}_{m}}$ denote the jump of the function $u$ across the common edge $\Gamma^{j}_m$. We assume that this edge is mapped from $\eta = -1$ in the reference element $S$ via the transformation $M_m : S \rightarrow \Omega_m$, and simultaneously from $\eta = 1$ in $S$ via the mapping $M_n : S \rightarrow \Omega_n$. Then $\left.{[u]}\right|_{\Gamma^{j}_m}$ is a function of $\xi$ only. Moreover, $\Gamma^{j}_m$ is the image of the interval $I=(-1,1)$ under these mappings.
By applying the rule of differentiation for composite functions, we obtain
\[(u)_{x_r}=(\hat{u})_{\xi}\xi_{x_r}+(\hat{u})_{\eta}\eta_{x_r},\]
for $r=1,2$. Then the jumps in the function $u$ and its derivatives across the inter element boundaries are defined by
\begin{align*}
\left\|[u]\right\|_{0,\Gamma^{j}_m}^{2}&=\left\|\hat{u}_{m}(\xi,-1)-\hat{u}_{n}(\xi,1)\right\|_{0,{S}}^{2}, \;\;\mbox{and}\\
\left\|[u_{x_{r}}]\right\|_{\frac{1}{2},\Gamma^{j}_m}^{2}&=\left\|\left((\hat{u}_{m})_{\xi}\xi_{x_r}+(\hat{u}_m)_{\eta}\eta_{x_r}\right)(\xi,-1)
-\left((\hat{u}_{n})_{\xi}\xi_{x_r}+(\hat{u}_n)_{\eta}\eta_{x_r}\right)(\xi,1)\right\|_{\frac{1}{2},{S}}^{2}.
\end{align*}
Similarly, if $\Omega_p$ and $\Omega_q$ denote two neighbouring elements in the sub-division of $\Omega$ and $\Gamma^{i}_{q}$ be
an edge of $\Omega_p$ and $\Gamma^{k}_q$ be an edge of $\Omega_q$ such that $\Gamma^{i}_p=\Gamma^{k}_q$, i.e. $\Gamma^{i}_p$
is an edge common to $\Omega_p$ and $\Omega_q$ which we assume is the image of $\xi=-1$, and $\xi=1$ under the mappings
$M_{p}:S\rightarrow\Omega_{p}$ and $M_{q}:S\rightarrow\Omega_{q}$ respectively. Then we define the jump terms,
\begin{align*}
\left\|[u]\right\|_{0,\Gamma^{i}_p}^{2}&=\left\|\hat{u}_{p}(-1,\eta)-\hat{u}_{q}(1,\eta)\right\|_{0,{S}}^{2},\;\;\mbox{and}\\
\left\|[u_{x_{s}}]\right\|_{\frac{1}{2},\Gamma^{i}_p}^{2}&=\left\|\left((\hat{u}_{p})_{\xi}\xi_{x_s}+(\hat{u}_p)_{\eta}\eta_{x_s}\right)(-1,\eta)
-\left((\hat{u}_{q})_{\xi}\xi_{x_s}+(\hat{u}_q)_{\eta}\eta_{x_s}\right)(1,\eta)\right\|_{\frac{1}{2},{S}}^{2},
\end{align*}
for $s=1,2$.

Let $\Pi^W(\Omega)$ denotes the space of polynomials of degree less than or equal to $W$ in each variable on $\Omega$ and therefore,
$\Pi^W(S)$ denotes the space of polynomials of degree less than or equal to $W$ in each variable on $S$. Let
$\{\mathcal{F}_{u}\}$ be the spectral element representation of the function $u$ on $\Omega_l$ i.e.,
\begin{align*}
\{\mathcal F_{u}\}=\left\{\{\hat{u}_{l}(\xi,\eta)\}_{l}\right\}\:,
\end{align*}
and let $\mathcal{S}^{N,W}\{\mathcal{F}_{u}\}$ denote the space of spectral element functions on the whole domain
$\Omega$. 

To present the stability theorem, we need a set of quadratic form $\mathcal{V}^{N,W}(\{\mathcal{F}_{u}\})$ and $\mathcal{U}^{N,W}(\{\mathcal{F}_{u}\})$ defined as follows:
\begin{align}\label{eq3.6}
\mathcal{V}^{N,W}(\{\mathcal{F}_{u}\})&=\sum_{l=1}^{N}\|L_l\hat{u}_l(\xi,\eta)\|^{2}_{0,S}
+\sum_{\Gamma^{j}_l\subseteq \bar{\Omega}\setminus\partial\Omega}\left(\|[u]\|^{2}_{0,\Gamma^{j}_l}+\sum_{k=1}^{2}\|[u_{x_k}]\|^{2}_{1/2,\Gamma^{j}_l}\right)
+\sum_{\Gamma^{s}\subseteq\partial\Omega}\left\|u\right\|^{2}_{\frac{3}{2},\Gamma^{s}},
\end{align}
\begin{align}\label{eq3.7}
\mathcal{U}^{N,W}(\{\mathcal{F}_{u}\})&=\sum_{l=1}^{N}\sum_{|\alpha|\leq2}\epsilon^{2|\alpha|}\|D^{\alpha}\hat{u}_l\|^2_{0,S}
\notag\\
&=\sum_{l=1}^{N}\int_{-1}^{1}\int_{-1}^{1}\Big(\epsilon^{4}\left((\hat{u}_l)_{\xi\xi}^2+2(\hat{u}_l)_{\xi\eta}^2+(\hat{u}_l)_{\eta\eta}^2\right)+\epsilon^{2}\left((\hat{u}_l)_{\xi}^2+(\hat{u}_l)_{\eta}^2\right)+\hat{u}_l^2\Big)d\xi d\eta.
\end{align}
Note that the quadratic form $\mathcal{U}^{N,W}(\{\mathcal{F}_{u}\})$ is a scaled $H^2$-norm of $\hat{u}$ on $S$
which accounts for the growth of the derivatives with respect to the boundary layer parameter $\epsilon$. We shall use
$\mathcal{U}^{N,W}(\{\mathcal{F}_{u}\})$ as a preconditioner for the quadratic form $\mathcal{V}^{N,W}(\{\mathcal{F}_{u}\})$
to solve the normal equations derive from the least-squares formulation as will be shown in section~\ref{sec5}.

\noindent
We now state and prove the main stability estimate theorem.
\begin{thm}[Stability estimate]\label{thm3}
Consider the elliptic boundary layer problem (\ref{eq2.6}) with Dirichlet boundary conditions (\ref{eq2.7}).
Then there exists a constant $C_\epsilon > 0$ (depending on $\epsilon$ but independent of $W$) such that
\begin{equation}\label{eq3.8}
\mathcal{U}^{N,W}(\{\mathcal{F}_{u}\})\leq C_\epsilon(\log W)^2 \mathcal{V}^{N,W}(\{\mathcal{F}_{u}\}),
\end{equation}
provided $W=\mathcal{O}\left(\frac{1}{\epsilon}\right)$.
\end{thm}
\begin{proof}
By the Lemma A.1 (see Appendix A), there exists $\{\{\hat{v}_{l}(\xi,\eta)\}_l\}$ (where $\hat{v}_{l}$, $l=1,2,\cdots,N$ is zero on $\Gamma$)
for which $w=u+v\in H^{2}(\Omega)$. Moreover $w=u$ on the domain boundary $\Gamma$. Using Minkowski inequality, we get
\begin{align}\label{eq3.9}
\sum_{l=1}^{N}\|\hat{u}_{l}\|^{2}_{2,S}&\leq C\left(\sum_{l=1}^{N}\|\hat{u}_{l}+\hat{v}_{l}\|^{2}_{2,S}+\sum_{l=1}^{N}\|\hat{v}_{l}\|^{2}_{2,S}\right),
\end{align}
Using regularity result stated in Theorem~\ref{thm1}, the estimate
\begin{align}\label{eq3.10}
\sum_{l=1}^{N}\|\hat{u}_{l}\|^{2}_{2,S}\leq & C_\epsilon\left(\sum_{l=1}^{N}\|{L}_{l}\hat{u}_{l}\|^{2}_{2,S}+\|w\|_{\frac{3}{2},\Gamma}^{2}+\sum_{l=1}^{N}\|\hat{v}_{l}\|^{2}_{2,S}\right).
\end{align}
holds. Substituting the obtained results from Lemma A.1 and Lemma A.2 in equation (\ref{eq3.10}), we obtain (\ref{eq3.8}).
\end{proof}
\begin{rem}\label{rem3.1}
We emphasize that the stability estimate in Theorem \ref{thm3} for the pure $p$-version on a fixed mesh is inherently non-uniform,
as reflected by the $\epsilon$-dependent constant $C_\epsilon$ and the asymptotic condition $W = \mathcal{O}\left(\frac{1}{\epsilon}\right)$.
In practical computations for singularly perturbed problems where $\epsilon$ can be extremely small (e.g., $\epsilon \approx 10^{-10}$),
scaling the polynomial degree $W$ inversely with $\epsilon$ is computationally infeasible. This severe practical limitation, and the
corresponding non-uniformity of the estimates, fundamentally motivates our formulation of the $rp$-version which employs specialized
boundary layer meshes in the subsequent sections. The $rp$-version successfully overcomes this restriction, eliminating the
$W = \mathcal{O}(1/\epsilon)$ requirement to achieve true parameter uniform convergence.
\end{rem}

\section{Numerical scheme}\label{sec4}

The construction of a spectral element approximation for $u$ on every element of the subdivision of $\Omega$ is detailed in Section~\ref{sec3}. To develop the numerical scheme, we first introduce a functional $\mathcal R^{N,W}({\mathcal F_u})$, which is associated with the quadratic form $\mathcal V^{N,W}({\mathcal F_u})$.

Let $J_{l}(\xi,\eta)$ be the Jacobian of the mapping $M_{l}(\xi,\eta)$. Let $f_{l}(\xi,\eta)=f_{l}\left(M_{l}(\xi,\eta)\right)$ and define
\[F_{l}(\xi,\eta)=f_{l}(\xi,\eta)\sqrt{J_{l}(\xi,\eta)},\]
for $l=1,2,\cdots,N$.

Recall that the boundary condition $w=g$ on $\Gamma$ in the discrete form will be $u=g$ on $\Gamma^{s}\cap\partial\Omega^{m}$, where $\Gamma^{s}\subseteq\partial\Omega$, $\partial\Omega^{m}$ denotes the boundary of $\Omega^{m}$.
Let $\Gamma^{s}\cap\partial\Omega^{m}=\gamma^{m}$ be the image of the mapping $M^{m}$ of $S$ onto $\Omega^{m}$ corresponding to
the $\xi=1$ and
\[h_{m}(\eta)=g\left(M^{m}(1,\eta)\right),\quad -1\leq\eta\leq 1.\]
Define the functional
\begin{align}
\mathcal{R}^{N,W}(\{\mathcal{F}_{u}\})&=\sum_{l=1}^{N}\|({L}_{l})\hat{u}_{l}(\xi,\eta)-F_{l}(\xi,\eta)\|^{2}_{0,S}
+\sum_{\Gamma_{l}^j\subseteq{\bar{\Omega}}\setminus\partial\Omega}\left(\|[u]\|^{2}_{0,\Gamma_{l}^j}+\sum_{k=1}^{2}\|[u_{x_k}]\|^{2}_{\frac{1}{2},\Gamma_{l}^j}\right) \notag\\
&+\sum_{\Gamma^{s}\subseteq \partial\Omega}\|u-h_{m}\|^{2}_{\frac{3}{2},\Gamma^{s}}.
\end{align}

Accordingly, the numerical scheme is presented as follows:
\begin{prop}
Find $\mathcal{F}_{z}\in {S}^{N,W}$ which minimizes the functional $\mathcal{R}^{N,W}(\{\mathcal{F}_{u}\})$ over all $\mathcal{F}_{u}\in{S}^{N,W}$,
where ${S}^{N,W}$ is the space of spectral element functions $\mathcal{F}_{u}$.
\end{prop}
Therefore, our approximate solution corresponds to the unique $\{\mathcal{F}_{z}\}\in S^{N,W}$, which minimize the functional
$\mathcal{R}^{N,W}(\{\mathcal F_{u}\})$ over all $\{\mathcal F_{u}\}\in S^{N,W}$. More precisely,

\emph{The numerical scheme seeks a solution that minimizes three contributions: the sum of the squared norm of the residuals in the partial differential equation, the sum of the squared norms of the residuals in the boundary conditions (measured in fractional Sobolev norms), and sum of the squared norms of the jumps in the function and its derivatives across inter-element boundaries, evaluated in appropriate fractional-order Sobolev norms. The latter term is introduced to enforce continuity along the inter-element boundaries.}

The proposed numerical scheme is essentially based on a least-squares method and we apply the preconditioned conjugate gradient method
(PCGM) to solve the normal equations which result from the least-squares formulation. Suppose that the normal equations are
\[A^TAU=A^Tb,\quad \text{or}\quad BU=c,\quad \text{with}\quad A^TA=B,\quad A^Tb=c.\]

In PCGM every iteration requires the computation of a matrix-vector product. The practicality of the method depends on the efficient and inexpensive evaluation of $BU$ for any vector $U$.

It can be shown as in~\cite{HUS2,HUS3} that $BU$ can be computed
economically without storing the mass and stiffness matrices explicitly unlike the traditional FEM. Furthermore, we should be able to develop a preconditioner for the matrix $B$ that is both effective and robust, ensuring that the condition number of the preconditioned matrix $M^{-1}B$ is as low as feasible. Once, this is done, we can use the PCGM to efficiently compute $(A^TA)^{-1}b$ for any vector $b$. In Section~\ref{sec5}, a preconditioner is constructed for the matrix associated with the normal equations, formulated in terms of a parameter-weighted $H_{\epsilon}^2$-norm.
 As shown in~\cite{DBR}, there exists a tensor-product polynomial basis that diagonalizes this preconditioner.
Since the least-squares functional contain jumps in function and its derivatives at the inter element boundaries, therefore the
values of the function and its derivatives must be exchanged between the elements in each iteration of the PCGM.

\section{Preconditioning techniques and parallelization}\label{sec5}

The quadratic form $\mathcal{V}^{N,W}(\{\mathcal F_{u}\})$ coincides with the functional  $\mathcal{R}^{N,W}(\{\mathcal F_{u}\})$ in the case of zero input data. Our goal is to design a quadratic form that is spectrally equivalent to the least-squares functional $\mathcal{V}^{N,W}(\{\mathcal F_{u}\})$,
computationally easy to invert, and ensures that the preconditioned system has a condition number independent of the mesh size (i.e., independent of the number of elements $N$), polynomial degree $W$, and boundary layer parameter $\epsilon$. This will guarantee that the preconditioner is maximally robust with respect to all the parameters.

To construct the preconditioner, consider the reference square $S=(-1,1)^2$ and the quadratic form $\mathcal{U}^{N,W}(\{\mathcal F_{u}\})$ as given in (\ref{eq3.7}). Now $\mathcal{U}^{N,W}(\{\mathcal F_{u}\}$ can be reformulated in the compact form as
\begin{align}\label{eq5.1}
\mathcal{U}^{N,W}(\{\mathcal F_{u}\})&=\epsilon^{4}|{\hat{u}}|^{2}_{2,S}+\epsilon^{2}|\hat{u}|^2_{1,S}+\|\hat{u}\|^2_{0,S}.
\end{align}

The preconditioner for the whole domain $\Omega$ corresponds to the quadratic form
\begin{align}\label{eq5.2}
\mathcal{U}^{N,W}(\{\mathcal F_{u}\})&=\sum_{l=1}^{N}\int_{S=M_l^{-1}(\Omega_l)}\left(\sum_{|\alpha|=2}\epsilon^4|D^{\alpha}\hat{u}_l|^{2}+\epsilon^{2}\sum_{|\alpha|=1}|D^{\alpha}\hat{u}_l|^2+|\hat{u}_l|^2\right)d\xi d\eta\notag\\
&=\sum_{l=1}^{N}\int_{S}\left(\epsilon^4(\hat{u}_l)_{\xi\xi}^2+2\epsilon^4(\hat{u}_l)_{\xi\eta}^2+\epsilon^4(\hat{u}_l)_{\eta\eta}^2+\epsilon^{2}(\hat{u}_l)_\xi^2+\epsilon^{2}(\hat{u}_l)_\eta^2+\hat{u}_l^2\right)d\xi d\eta.
\end{align}
The following spectral equivalence holds.
\begin{thm}\label{thm5.1}
The quadratic form $\mathcal{U}^{N,W}(\{\mathcal F_{u}\})$ is spectrally equivalent to the quadratic form $\mathcal{V}^{N,W}(\{\mathcal F_{u}\})$. Moreover, there exist positive constants $c_\epsilon$ and $C_\epsilon$ (depending on $\epsilon$, but independent of $N$ and $W$), such that
\begin{align}\label{eq5.3}
c_\epsilon\mathcal{V}^{N,W}(\{\mathcal F_{u}\})\leq \mathcal{U}^{N,W}(\{\mathcal F_{u}\})\leq C_\epsilon(\log W)^2\mathcal{V}^{N,W}(\{\mathcal F_{u}\}),
\end{align}
provided $W = \mathcal{O}\left(\frac{1}{\epsilon}\right)$.
\end{thm}

\begin{proof}
From the stability estimate in Theorem~\ref{thm3}, there exists a positive constant $C_\epsilon$ (depending on $\epsilon$) independent of $N$ and $W$ such that the upper bound
\begin{align}\label{eq5.4}
\mathcal{U}^{N,W}(\{\mathcal F_{u}\})\leq C_\epsilon(\log W)^2\mathcal{V}^{N,W}(\{\mathcal F_{u}\}),
\end{align}
holds, provided $W = \mathcal{O}\left(\frac{1}{\epsilon}\right)$.

For the proof of the lower bound, we need to bound the norm of the differential operator, the jump terms and boundary terms appearing in $\mathcal{V}^{N,W}(\{\mathcal{F}_u\})$
in terms of $\mathcal{U}^{N,W}(\{\mathcal{F}_u\})$. For the residual term, using the triangle inequality and Young's inequality, we have
\begin{align*}
\|L{u}_l\|^2_{0,\Omega_l}&\leq\|-\epsilon^2\Delta{u}_l+au_l\|^2_{0,\Omega_l}\leq 2\epsilon^4\|\Delta{u}_l\|_{0,\Omega_l}^2+2\|au\|_{0,\Omega_l}^2\notag\\
&\leq 2\epsilon^4\|\Delta{u}_l\|_{0,\Omega_l}^2+2\|a\|_{L^\infty(\Omega_l)}^2\|u\|_{0,\Omega_l}^2.
\end{align*}
Summing over all elements,
\begin{align}\label{eq5.5}
\sum_{l=1}^{N}\|L{u}_l\|^2_{0,\Omega_l}&\leq 2\epsilon^4\sum_{l=1}^{N}\|\Delta{u}_l\|^2_{0,\Omega_l}+\sum_{l=1}^{N}2\|a\|^2_{L^\infty(\Omega_l)}\|u\|^2_{0,\Omega_l}\leq C \mathcal{U}^{N,W} (\{\mathcal{F}_u\}).
\end{align}
Next, we consider the jumps in the derivatives to enforce the discontinuity across inter-element edges. Let $\Gamma_l^j$ be an edge in
the interior of an element $\Omega_l$. Using trace theorems and Sobolev embeddings~\cite{ADAM},
\begin{align*}
\|\left[u_{x_k}\right]\|^2_{1/2,\Gamma_l^j}\leq C\|\Delta u\|_{0,S}^2,\quad \|[u]\|^2_{0,\Gamma_l^j}\leq C\|u\|_{0,S}^2,
\end{align*}
where $C$ is a constant independent of $\epsilon$ and $W$. Because $\mathcal{U}^{N,W}$ is an $\epsilon$-scaled norm (weighting second derivatives by $\epsilon^4$), bounding the unscaled $H^2$-norm by $\mathcal{U}^{N,W}$ introduces an $\epsilon$-dependence. Thus, there exists a constant $C_\epsilon' = \mathcal{O}(\epsilon^{-4})$ such that
\begin{align}\label{eq5.6}
\sum_{\Gamma_l^j\subset\bar{\Omega}\setminus\partial\Omega}\left(\|[u]\|^2_{0,\Gamma_l^j}+\sum_{k=1}^{2}\|\left[u_{x_k}\right]\|^2_{1/2,\Gamma_l^j}\right)
&\leq C\left(\|u\|_{0,S}^2+\|\Delta u\|_{0,S}^2\right)\leq C \|u\|_{2,S}^2\leq C_\epsilon' \mathcal{U}^{N,W} (\{\mathcal{F}_u\}).
\end{align}
Finally, if $\Gamma^s$ is an edge that lies on the domain boundary $\partial\Omega$, then using the trace theorem for Sobolev spaces~\cite{ADAM}, we have
\begin{align*}
\|u_l\|^2_{0,\Gamma^s}\leq \|u_l\|^2_{\frac{3}{2},S}\lesssim \|u_l\|^2_{2,S}.
\end{align*}
Summing over all $\Gamma^s\subseteq\partial\Omega$ and again utilizing the $\epsilon$-dependent bound for the unscaled norm,
\begin{align}\label{eq5.7}
\sum_{\Gamma^s\subseteq\partial\Omega}\|u\|^2_{\frac{3}{2},\Gamma^s}\leq C\|u\|_{2,S}^2 \leq C_\epsilon' \mathcal{U}^{N,W} (\{\mathcal{F}_u\}).
\end{align}
Adding and combining (\ref{eq5.5})$-$(\ref{eq5.7}), there exists a generic constant $C_\epsilon''$ (depending on $\epsilon$, but independent of $N$ and $W$) such that
\begin{align}\label{eq5.8}
\mathcal{V}^{N,W}(\{\mathcal F_{u}\}) &\leq C_\epsilon'' \mathcal{U}^{N,W} (\{\mathcal{F}_u\}).
\end{align}
Combining the upper and lower bounds given by (\ref{eq5.4}) and (\ref{eq5.8}) respectively, we get
\[c_\epsilon \mathcal{V}^{N,W} (\{\mathcal F_{u}\})\leq \mathcal{U}^{N,W}(\{\mathcal F_{u}\}) \leq C_\epsilon (\log W)^2 \mathcal{V}^{N,W}(\{\mathcal F_{u}\}),\]
with $c_\epsilon=\frac{1}{C_\epsilon''}$. Hence, the quadratic forms $\mathcal{U}^{N,W}(\{\mathcal F_{u}\})$ and $\mathcal{V}^{N,W}(\{\mathcal F_{u}\})$ are
spectrally equivalent.
\end{proof}

The quadratic form $\mathcal{U}^{N,W}(\{\mathcal F_{u}\})$, which is composed of a decoupled set of quadratic forms on each element, can be employed as a preconditioner for the matrix $A$. As a result, the preconditioned system exhibits a condition number of $O((\log W)^2)$. Thus on each element, the preconditioner is associated locally with the the quadratic form
\begin{align}\label{eq5.9}
\mathcal B_\epsilon({u})=\int_{S}\left(\epsilon^4u_{\xi\xi}^2+2\epsilon^4u_{\xi\eta}^2+\epsilon^4u_{\eta\eta}^2+u_\xi^2+u_\eta^2+u^2\right)d\xi d\eta,
\end{align}
where $u=u(\xi,\eta)$ is a polynomial of degree $W$ in $\xi$ and $\eta$ separately.

Define a quadratic form
\begin{align}\label{eq5.10}
\mathcal{C}_\epsilon({u})=\int_{S}\left(\epsilon^4u_{\xi\xi}^2+\epsilon^4u_{\eta\eta}^2+u_\xi^2+u_\eta^2+u^2\right)d\xi d\eta.
\end{align}

According to~\cite{DBR} (Theorem 2.1), $\mathcal C_\epsilon({u})$ and $\mathcal B_\epsilon({u})$ are spectrally equivalent. Furthermore, there is a basis of polynomials that diagonalizes $\mathcal{C}_\epsilon(u)$ using separation of variables, making the associated matrix straightforward to invert on each element (see~\cite{DBR} for details).
For ease of parallelism, each element of the discretized domain is assigned to a unique processor. While applying the PCGM technique, communication between neighboring processors is limited to the exchange of function values and their derivatives at inter-element boundaries. Additionally, only two global scalar values are required to update the approximate solution and the search direction. As a result, inter-processor communication remains minimal. To attain exponential accuracy in the numerical solution, approximately $O(W \ln W)$ iterations are required by the PCGM.

\section{Error estimates}\label{sec6}
We now present error bounds for the spectral element approximation of the problem under consideration. We first study the $p$-version approach of the method by keeping the number of elements $N=1$ and vary the polynomial order $W$. Let $S^W$ denote the space of spectral element functions for the $p$-version. The main result is as follows:

\begin{thm}\label{thm6.1}
Let $\Omega=S=(-1,1)^2$, $I=(-1,1)$ and $U(x,y)=w\left(M(\xi,\eta)\right)=\phi(x)e^{-\frac{y+1}{\epsilon}}$ for $(\xi,\eta)\in S$
be the exact solution of the problem (\ref{eq2.6})-(\ref{eq2.7}) such that $\phi(x)\in H^2(I)$ is independent of $\epsilon$ and let
$\mathcal{F}_{z}\in S^W$ minimizes the functional $\mathcal{R}^{W}(\{\mathcal{F}_{u}\})$ over all $\mathcal{F}_{u}\in S^W$. Then there
exist a constant $C$ (independent of $\epsilon$ and $W$) such that the error estimate
\begin{align}\label{eq6.1}
\|U(x,y)-z(x,y)\|_{2,\epsilon,\Omega}\leq C \frac{\sqrt{\log W}}{W},
\end{align}
holds provided $W=O(\frac{1}{\epsilon})$.
\end{thm}
\begin{proof}
Recall from (\ref{eq2.5}) that,
\begin{align*}
\|u\|^2_{m,\epsilon,\Omega}=\sum_{|\alpha|\leq m}\epsilon^{2|\alpha|}\|\partial^\alpha u\|^2_{0,\Omega}=\epsilon^{4}|u|_{2,\Omega}^2
+\epsilon^{2}|u|_{1,\Omega}^2+\|u\|_{0,\Omega}^2.
\end{align*}
Thus, to establish (\ref{eq6.1}) we need to get bounds for the terms,
\[\|U(x,y)-z(x,y)\|^{2}_{0,\Omega},\quad|U(x,y)-z(x,y)|^{2}_{1,\Omega},\quad\text{and}\quad |U(x,y)-z(x,y)|^{2}_{2,\Omega},\]
where $z(x,y)\in S^W$ is the unique spectral element minimizer of the functional $\mathcal{R}^{W}(\{\mathcal{F}_{u}\})$
in the space of spectral element functions $S^W$ on $\Omega$.

Let $z(x,y)=h(x)k(y)$ and $U(x,y)=\phi(x)w_\epsilon(y)$, where $w_\epsilon(y)=e^{-\frac{y+1}{\epsilon}}$ be as defined in (\ref{eq2.9}),
then
\begin{align}
\|U(x,y)-z(x,y)\|^{2}_{0,\Omega}&=\left\|\phi(x)w_\epsilon(y)-h(x)k(y)\right\|^{2}_{0,\Omega}\notag\\
&=\left\|\phi(x)w_\epsilon(y)-\phi(x)k(y)+\phi(x)k(y)-h(x)k(y)\right\|^{2}_{0,\Omega}.\notag
\end{align}
Let $f_1(x,y)=\phi(x)\left(w_\epsilon(y)-k(y)\right)$ and $f_2(x,y)=k(y)\left(\phi(x)-h(x)\right)$, then
\begin{align}
\|U(x,y)-z(x,y)\|^{2}_{0,\Omega}&=\left\|f_1+f_2\right\|^{2}_{0,\Omega}\leq 2\left\|f_1\right\|^{2}_{0,\Omega}+2\left\|f_2\right\|^{2}_{0,\Omega}.\notag
\end{align}
Using Fubini's theorem in the two terms on RHS, we get
\begin{align}\label{eq6.2}
\|U(x,y)-z(x,y)\|_{0,\Omega}^2&=2\|\phi\|_{0,I}^2\left\|w_\epsilon-k\right\|_{0,I}^2+2\|k\|_{0,I}^2\left\|\phi-h\right\|_{0,I}^2.
\end{align}
In the same way, we have
\begin{align}\label{eq6.3}
|U(x,y)-z(x,y)|^{2}_{1,\Omega}&=\left\|\frac{\partial U}{\partial x}-\frac{\partial z}{\partial x}\right\|^{2}_{0,\Omega}
+\left\|\frac{\partial U}{\partial y}-\frac{\partial z}{\partial y}\right\|^{2}_{0,\Omega}\notag\\
&=\left\|\phi^\prime(x)w_\epsilon(y)-h^\prime(x)k(y)\right\|_{0,\Omega}^2+\left\|\phi(x)w^\prime_\epsilon(y)-h(x)k^\prime(y)\right\|^{2}_{0,\Omega}\notag\\
&=\left\|\phi^\prime(x)w_\epsilon(y)-\phi^\prime(x)k(y)+\phi^\prime(x)k(y)-h^\prime(x)k(y)\right\|_{0,\Omega}^2\notag\\
&+\left\|\phi(x)w^\prime_\epsilon(y)-\phi(x) k^\prime(y)+\phi(x)k^\prime(y)-h(x)k^\prime(y)\right\|^{2}_{0,\Omega}\notag\\
&\leq2\|\phi^\prime\|_{0,I}^2\|w_{\epsilon}-k\|_{0,I}^2+2\|k\|_{0,I}^2\|\phi^{\prime}-h^{\prime}\|_{0,I}^2\notag\\
&+2\|\phi\|_{0,I}^2\|w_{\epsilon}^{\prime}-k^{\prime}\|_{0,I}^2+2\|k^{\prime}\|_{0,I}^2\|\phi-h\|_{0,I}^2,
\end{align}
and
\begin{align}\label{eq6.4}
|U(x,y)-z(x,y)|^{2}_{2,\Omega}&=\left\|\frac{\partial^2 U}{\partial x^2}-\frac{\partial^2 z}{\partial x^2}\right\|^{2}_{0,\Omega}
+\left\|\frac{\partial^2 U}{\partial y^2}-\frac{\partial^2 z}{\partial y^2}\right\|^{2}_{0,\Omega}+\left\|\frac{\partial^2 U}{\partial x \partial y}-\frac{\partial^2 z}{\partial x \partial y}\right\|^{2}_{0,\Omega}\notag\\
&=\left\|\phi^{\prime\prime}(x)w_\epsilon(y)-h^{\prime\prime}(x)k(y)\right\|_{0,\Omega}^2+\left\|\phi(x)w^{\prime\prime}_\epsilon(y)-h(x)k^{\prime\prime}(y)\right\|^{2}_{0,\Omega}\notag\\
&+\left\|\phi^{\prime}(x)w_\epsilon^{\prime}(y)-h^{\prime}(x)k^{\prime}(y)\right\|_{0,\Omega}^2\notag\\
&=\left\|\phi^{\prime\prime}(x)w_\epsilon(y)-\phi^{\prime\prime}(x)k(y)+\phi^{\prime\prime}(x)k(y)-h^{\prime\prime}(x)k(y)\right\|_{0,\Omega}^2\notag\\
&+\left\|\phi(x)w^{\prime\prime}_\epsilon(y)-\phi(x) k^{\prime\prime}(y)+\phi(x)k^{\prime\prime}(y)-h(x)k^{\prime\prime}(y)\right\|^{2}_{0,\Omega}\notag\\
&+\left\|\phi^{\prime}(x)w_\epsilon^{\prime}(y)-\phi^{\prime}(x)k^{\prime}(y)+\phi^{\prime}(x)k^{\prime}(y)-h^{\prime}(x)k^{\prime}(y)\right\|_{0,\Omega}^2\notag\\
&\leq2\|\phi^{\prime\prime}\|_{0,I}^2\|w_{\epsilon}-k\|_{0,I}^2+2\|k\|_{0,I}^2\|\phi^{\prime\prime}-h^{\prime\prime}\|_{0,I}^2\notag\\
&+2\|\phi\|_{0,I}^2\|w_{\epsilon}^{\prime\prime}-k^{\prime\prime}\|_{0,I}^2+2\|k^{\prime\prime}\|_{0,I}^2\|\phi-h\|_{0,I}^2\notag\\
&+2\|\phi^{\prime}\|_{0,I}^2\|w_{\epsilon}^{\prime}-k^{\prime}\|_{0,I}^2+2\|k^{\prime}\|_{0,I}^2\|\phi^{\prime}-h^{\prime}\|_{0,I}^2.
\end{align}
Let $h(x)\in\Pi_W(I)$ be the $H^2$ projection of $\phi(x)$ onto $\Pi_W(I)$, and let
\begin{align}\label{eq6.5}
\phi_m^W := \left\|\frac{d^m}{dx^m}\left(\phi(x)-h(x)\right)\right\|_{0,I}=\left\| \phi^{(m)} - h^{(m)} \right\|_{0,I}.
\end{align}
denotes the projection error in the $H^m$-seminorm for $0\leq m\leq 2$. Moreover, for any norm $\|\cdot\|$ (and $0\leq m\leq 2$) using the triangle inequality, we have
\begin{align}\label{eq6.6}
\left\|\frac{d^m k}{dy^m}\right\|\leq \left\|\frac{d^m w_\epsilon}{dy^m}\right\|+\left\|\frac{d^m k}{dy^m}-\frac{d^m w_\epsilon}{dy^m}\right\|.
\end{align}
Combining (\ref{eq6.2}$)-$(\ref{eq6.4}) and using (\ref{eq6.5}, (\ref{eq6.6}) we obtain
\begin{align}\label{eq6.7}
\|U(x,y)-z(x,y)\|^{2}_{2,\epsilon,\Omega}&=\epsilon^4|U(x,y)-z(x,y)|^{2}_{2,\Omega}+\epsilon^2|U(x,y)-z(x,y)|^{2}_{1,\Omega}+\|U(x,y)-z(x,y)\|^{2}_{0,\Omega}\notag\\
&\leq2\epsilon^4\|\phi^{\prime\prime}\|_{0,I}^2\|w_{\epsilon}-k\|_{0,I}^2+2\epsilon^4\left(\|w_\epsilon\|_{0,I}+\left\|w_\epsilon-k\right\|_{0,I}\right)^2\left(\phi_2^W\right)^2\notag\\
&+2\epsilon^4\|\phi\|_{0,I}^2\|w_{\epsilon}^{\prime\prime}-k^{\prime\prime}\|_{0,I}^2+2\epsilon^4\left(\|w^{\prime\prime}_\epsilon\|_{0,I}+\left\|w^{\prime\prime}_\epsilon-k^{\prime\prime}\right\|_{0,I}\right)^2\left(\phi_0^W\right)^2\notag\\
&+2\epsilon^4\|\phi^{\prime}\|_{0,I}^2\|w_{\epsilon}^{\prime}-k^{\prime}\|_{0,I}^2+2\epsilon^4\left(\|w^{\prime}_\epsilon\|_{0,I}+\left\|w^{\prime}_\epsilon-k^{\prime}\right\|_{0,I}\right)^2\left(\phi_1^W\right)^2\notag\\
&+2\epsilon^2\|\phi^\prime\|_{0,I}^2\|w_{\epsilon}-k\|_{0,I}^2+2\epsilon^2\left(\|w_\epsilon\|_{0,I}+\left\|w_\epsilon-k\right\|_{0,I}\right)^2\left(\phi_1^W\right)^2\notag\\
&+2\epsilon^2\|\phi\|_{0,I}^2\|w_{\epsilon}^{\prime}-k^{\prime}\|_{0,I}^2+2\epsilon^2\left(\|w^{\prime}_\epsilon\|_{0,I}+\left\|w^{\prime}_\epsilon-k^{\prime}\right\|_{0,I}\right)^2\left(\phi_{0}^W\right)^2\notag\\
&+2\|\phi\|_{0,I}^2\left\|w_\epsilon-k\right\|_{0,I}^2+2\left(\|w_\epsilon\|_{0,I}+\left\|w_\epsilon-k\right\|_{0,I}\right)^2\left(\phi_0^W\right)^2\notag\\
&\leq C \Big[\left( \|\phi\|_{0,I}^2 + \epsilon^2 \|\phi^{\prime}\|_{0,I}^2 + \epsilon^4 \|\phi^{\prime\prime}\|_{0,I}^2 \right) \|w_\epsilon-k\|_{2,I}^2 \notag\\
&+\left(\left(\phi_{0}^W\right)^2 + \epsilon^2 \left(\phi_{1}^W\right)^2 + \epsilon^4 \left(\phi_{2}^W\right)^2 \right)\left(\|w_\epsilon\|_{0,I}+\left\|w_\epsilon-k\right\|_{0,I}\right)^2 \notag\\
&+ \left(\epsilon^2 \left(\phi_{0}^W\right)^2 + \epsilon^4 \left(\phi_{1}^W\right)^2\right) \left(\|w^{\prime}_\epsilon\|_{0,I}+\left\|w^{\prime}_\epsilon-k^{\prime}\right\|_{0,I}\right)^2 \notag \\
&+\epsilon^4 \left(\phi_0^W\right)^2 \left(\|w^{\prime\prime}_\epsilon\|_{0,I}+\left\|w^{\prime\prime}_\epsilon-k^{\prime\prime}\right\|_{0,I}\right)^2\Big].
\end{align}
Now,
\[\|\phi\|_{0,I}^2 + \epsilon^2 \|\phi^\prime\|_{0,I}^2 + \epsilon^4 \|\phi^{\prime\prime}\|_{0,I}^2 \leq (1 + \epsilon^2 + \epsilon^4) \|\phi\|_{2,I}^2.\]
Using above estimate and Young's inequality in (\ref{eq6.7}), we obtain
\begin{align}\label{eq6.8}
\|U(x,y)-z(x,y)\|^{2}_{2,\epsilon,\Omega}&\leq C \Big[(1 + \epsilon^2 + \epsilon^4) \|\phi\|_{2,I}^2 \|w_\epsilon - k\|_{2,I}^2 \notag\\
&+\left(\left(\phi_{0}^W\right)^2 + \epsilon^2 \left(\phi_{1}^W\right)^2 + \epsilon^4 \left(\phi_{2}^W\right)^2 \right)\left(\|w_\epsilon\|_{0,I}^2+\left\|w_\epsilon-k\right\|_{0,I}^2\right) \notag\\
&+ \left(\epsilon^2 \left(\phi_{0}^W\right)^2 + \epsilon^4 \left(\phi_{1}^W\right)^2\right) \left(\|w^{\prime}_\epsilon\|_{0,I}^2+\left\|w^{\prime}_\epsilon-k^{\prime}\right\|_{0,I}^2\right) \notag \\
&+\epsilon^4 \left(\phi_0^W\right)^2 \left(\|w^{\prime\prime}_\epsilon\|_{0,I}^2+\left\|w^{\prime\prime}_\epsilon-k^{\prime\prime}\right\|_{0,I}^2\right)\Big].
\end{align}
Rearranging the terms and using the scaled $H^2$ Sobolev norm defined in (\ref{eq2.5}), we get
\begin{align}
\|U(x,y)-z(x,y)\|^{2}_{2,\epsilon,\Omega}& \leq C \Big[(1+\epsilon^2+\epsilon^4) \|\phi\|^2_{2,I} \|w_\epsilon(y) - k(y)\|^2_{2,I}\notag\\
&+\sum_{m=0}^2 \epsilon^{2m} (\phi_m^W)^2 \left(\|w_\epsilon(y)\|_{2,\epsilon,I}^2+\|w_\epsilon(y)-k(y)\|_{2,I}^2\right)\Big].\notag
\end{align}
Now, we have $\|w_\epsilon(y) - k(y)\|^2_{2,I}\leq C\epsilon^{-4}\|w_\epsilon(y) - k(y)\|^2_{2,\epsilon,I}$ and $\|w_\epsilon(y)\|^2_{2,I}\leq C\epsilon^{-4}\|w_\epsilon(y)\|^2_{2,\epsilon,I}$. Therefore,
the above estimate becomes,
\begin{align}\label{eq6.9}
\|U(x,y)-z(x,y)\|^{2}_{2,\epsilon,\Omega}& \leq C \epsilon^{-4}\Big[\|\phi\|^2_{2,I} \|w_\epsilon(y) - k(y)\|^2_{2,\epsilon,I}\notag\\
&+\sum_{m=0}^2 \epsilon^{2m} (\phi_m^W)^2 \left(\|w_\epsilon(y)\|_{2,\epsilon,I}^2+\|w_\epsilon(y)-k(y)\|_{2,\epsilon,I}^2\right)\Big],
\end{align}
where, $C$ is a generic constant independent of $\epsilon$ and $W$. Using approximation results from~\cite{SCHW1} and~\cite{XENO1} (Theorem 2.3.3, Chapter 2),
there exists a constant $C_1$ independent of $\epsilon$ and $W$ such that
\begin{align}\label{eq6.10}
\|w_\epsilon(y)-k(y)\|_{2,\epsilon,I} = \left\|e^{-\frac{(y+1)}{\epsilon}} - k(y)\right\|_{2,\epsilon,(-1,1)}\leq C_1 \frac{\sqrt{\log W}}{W}.
\end{align}
Thus, (\ref{eq6.9}) takes the form
\begin{align}\label{eq6.11}
\|U(x,y)-z(x,y)\|^{2}_{2,\epsilon,\Omega}& \leq C \epsilon^{-4}\Bigg[\|\phi\|^2_{2,I} \frac{\log W}{W^2}+\sum_{m=0}^2 \epsilon^{2m}(\phi_m^W)^2 \left(\|w_\epsilon(y)\|_{2,\epsilon,I}^2+\frac{\log W}{W^2}\right)\Bigg].
\end{align}
Also, $\phi(x)\in H^2(I)$, therefore $\|\phi(x)\|_{2,I}\leq C_2$ for some constant $C_2$ and by standard approximation results
for projection or interpolation of sufficiently smooth functions, the spectral approximation estimate
\begin{align}\label{eq6.12}
\phi_m^W =\left\| \phi^{(m)} - h^{(m)} \right\|_{0,I}\leq CW^{m-s} \|\phi\|_{s,I},\quad 0\leq m\leq s,
\end{align}
holds. From the above, for $s=2$ we get
\[\phi_0^W \leq K_0W^{-2} \|\phi\|_{2,I},\quad \phi_1^W \leq K_1W^{-1} \|\phi\|_{2,I},\quad \text{and}\quad \phi_2^W \leq K_2\|\phi\|_{2,I},\]
for some constants $K_m$, $m=0,1,2$. So that
\begin{align}\label{eq6.13}
\phi_0^W \leq L_0W^{-2},\quad \phi_1^W \leq L_1W^{-1},\quad \text{and}\quad \phi_2^W \leq L_2,
\end{align}
with $L_m=C_2K_m$ for all $0\leq m\leq 2$. Also, it is easy to see that
\begin{align}\label{eq6.14}
\|w_\epsilon(y)\|_{2,\epsilon,I}^2\leq \|w_\epsilon(y)\|_{2,I}^2\leq \frac{3\epsilon}{2}\leq 2.
\end{align}
Using (\ref{eq6.13}), (\ref{eq6.14}) in (\ref{eq6.11}) and noting that $\frac{\log W}{W^2}=O(1)$ for $W$ large enough, we obtain
\begin{align}\label{eq6.15}
\|U(x,y) - z(x,y)\|^2_{2,\epsilon,\Omega}
&\leq C \epsilon^{-4}\Bigg[C_2^2\cdot \frac{\log W}{W^2}+ \left(\frac{L_0^2}{W^{4}} + \epsilon^2 \frac{L_1^2}{W^{2}} + \epsilon^4 L_2^2 \right) \left( 2 + \frac{\log W}{W^2} \right)\Bigg]\notag\\
&\leq C\epsilon^{-4}\left[\frac{\log W}{W^2}+\left(\epsilon^4 +\frac{\epsilon^2}{W^{2}} + \frac{1}{W^4}\right)\left(2+\frac{\log W}{W^2} \right)\right]\notag\\
&\leq C\frac{\log W}{W^2}.
\end{align}
provided $W=O(\frac{1}{\epsilon})$. Taking square root in (\ref{eq6.15}), we get the desired error estimate (\ref{eq6.1}).
\end{proof}

\begin{rem}
The non-conforming spectral element solution be computed by applying PCGM to the normal equations. A subsequent correction produces a conforming solution $z_{\text{cor}}$, for which the error in the scaled $H^1$-norm is exponentially small in $W$. Such corrections are in line with~\cite{DHMU1,Tomar}, and the error estimate
\[\|U(x,y) - z_{\text{cor}}(x,y)\|_{1,\epsilon,\Omega}\leq C \frac{\sqrt{\log W}}{W}\]
holds, where $C$ is a constant independent of $\epsilon$ and $W$ and $z_{\text{cor}}$ is the corrected solution.
\end{rem}

\begin{rem}
The condition $W = \mathcal{O}(1/\epsilon)$ in Theorem \ref{thm6.1} implies that the standard $p$-version on a fixed mesh is inherently not $\epsilon$-uniform. Although the constant $C$ is mathematically independent of $\epsilon$ under this strict asymptotic requirement, resolving the layer as $\epsilon \to 0$ requires an unbounded increase in the polynomial degree $W$. Consequently, computing with a fixed $W$ naturally leads to degraded accuracy for small $\epsilon$. This limitation directly motivates the $rp$-version proposed in Theorem \ref{thm6.2}, which employs boundary layer meshes to recover true parameter-uniform convergence.
\end{rem}

We now study the $rp$-version of the method, i.e., we impose a spectral boundary layer mesh on $\Omega=(-1,1)^2$ by fixing
$N=2$ and vary $W$ on each element simultaneously in the mesh as follows. Define
\[\Omega_1=(-1,1)\times(-1,-1+\tau),\quad \tau=\kappa\epsilon W,\quad \kappa\in\mathbb{R}, \quad \Omega_2=\Omega\setminus\Omega_1\]
so that
\[\Omega=\Omega_1\cup\Omega_2\]
as shown in Figure \ref{fig3}. The parameter $\tau$ is called \emph{transition point} and $\kappa>0$ being a constant independent of $W$
and $\epsilon$ that needs to be chosen carefully.
\begin{figure}[ht]
\centering
\includegraphics[width=0.37\textwidth]{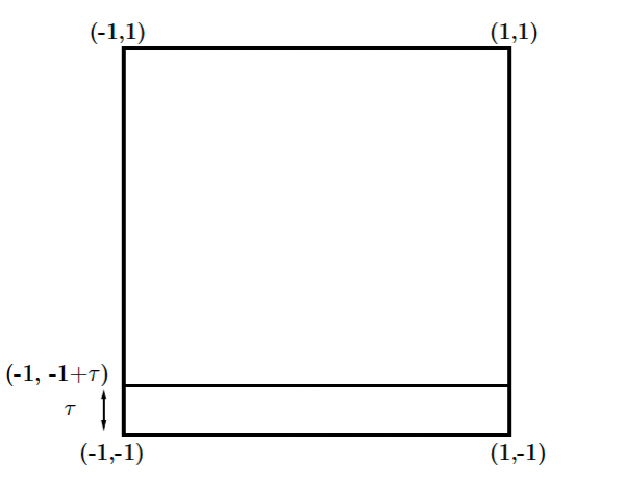}
\caption{A spectral boundary layer mesh on $\Omega=(-1,1)^2$.}
\label{fig3}
\end{figure}
\begin{thm}\label{thm6.2}
Consider the spectral boundary layer mesh on $\Omega=S=(-1,1)^2$ defined by $\Omega_1=(-1,1)\times(-1,\tau),\:\kappa\in\mathbb{R}$,
$\Omega_2=\Omega\setminus\Omega_1$ and let $U(x,y)=w\left(M(\xi,\eta)\right)=\phi(x)e^{-\frac{y+1}{\epsilon}}$ for $(\xi,\eta)\in S$
be the exact solution corresponding to the problem (\ref{eq2.6})-(\ref{eq2.7}). Suppose that $\phi(x)$ is analytic, independent of $\epsilon$ and let
$\mathcal{F}_{z}\in S^W$ minimizes the functional $\mathcal{R}^{W}(\{\mathcal{F}_{u}\})$ over all $\mathcal{F}_{u}\in S^W$. Then there
exist a constant $C$ (independent of $\epsilon$ and $W$) such that the least-squares error estimate
\begin{align}\label{eq6.16}
\|U(x,y)-z(x,y)\|_{2,\epsilon,\Omega}^2\leq C {\epsilon}\alpha^{2W_0},
\end{align}
holds, where $W_0=W+\frac{1}{2}$ and $\alpha(\in\mathbb{R})<1$ is a constant.
\end{thm}
\begin{proof}
Let $h(x)\in\Pi^W(I)$ be the best approximation of $\phi(x)\in H^2(I)$ in the $H^2$-projection sense and let $k(y)\in\Pi^W(I)$ be the
best $H^2$-approximation of $w_\epsilon(y)=e^{-\frac{y+1}{\epsilon}}$ on $I=(-1,1)$. Define
\[h(x,y)=h(x)k(y)\in\Pi^W(I)\otimes\Pi^W(I).\]
Now,
\begin{align}\label{eq6.17}
\|U(x,y)-z(x,y)\|_{2,\epsilon,\Omega}^2&=\|\phi(x)w_\epsilon(y)-h(x)k(y)\|_{2,\epsilon,\Omega}^2 \notag\\
&=\left\|\phi(x)\left(w_\epsilon(y)-k(y)\right)+\left(\phi(x)-h(x)\right)k(y)\right\|_{2,\epsilon,\Omega}^2 \notag\\
&\leq 2\left\|\phi(x)\left(w_\epsilon(y)-k(y)\right)\right\|_{2,\epsilon,\Omega}^2+2\left\|\left(\phi(x)-h(x)\right)k(y)\right\|_{2,\epsilon,\Omega}^2
\end{align}
By exponential approximation results for analytic boundary layer functions (cf. Canuto~\cite{CHQ2} and Schwab~\cite{SCHW}), there exists
a constant $C>0$ such that:
\begin{align*}
\|w_\epsilon(y)-k(y)\|_{m,\Omega}^2\leq C\epsilon\alpha^{2W_0},\quad m=0,1,2.
\end{align*}
Therefore, the first term on RHS in (\ref{eq6.17}) can be rewritten as
\begin{align}\label{eq6.18}
\left\|\phi(x)\left(w_\epsilon(y)-k(y)\right)\right\|_{2,\epsilon,\Omega}^2&\leq C \|\phi(x)\|_{2,I}^2\:\epsilon\alpha^{2W_0}.
\end{align}
As in Theorem~\ref{thm6.1}, the second terms can be written as,
\begin{align}\label{eq6.19}
\left\|\left(\phi(x)-h(x)\right)k(y)\right\|_{2,\epsilon,\Omega}^2&\leq \sum_{m=0}^2 \epsilon^{2m} (\phi_m^W)^2 \:\|k(y)\|_{m,I}^2\notag\\
&\leq C \left(\|\phi\|_{m,I}^2\right)\left(\epsilon^{4}+\epsilon^{2}W^{-2}+W^{-4}\right).
\end{align}
Choosing $W=O(\frac{1}{\epsilon})$, and using the fact that $\phi$ is analytic, it follows that there is a constant $C$ independent
of $\epsilon$ and $W$ such that
\begin{align}\label{eq6.20}
\left\|\left(\phi(x)-h(x)\right)k(y)\right\|_{2,\epsilon,\Omega}^2&\leq C \:\epsilon\alpha^{2W_0}.
\end{align}
Combining (\ref{eq6.18}) and (\ref{eq6.20}), we obtain the result from (\ref{eq6.17}).
\end{proof}
Theorem~\ref{thm6.2} guarantees that in order to resolve boundary layers at a robust exponential rate for the problem (\ref{eq2.6})$-$(\ref{eq2.7})
it is enough to use just two elements of the type $\Omega_1$ and $\Omega_2$ i.e., a thin strip (element) $\Omega_1$ near the boundary layer in
which the solution is approximated with a large polynomial order $W\geq O\left(\frac{1}{\epsilon}\right)$ and a coarse element $\Omega_2$ away
from the boundary in which the solution is approximated by polynomials having order $O(1)$. The solution of (\ref{eq2.6})$-$(\ref{eq2.7}) usually
have other (smooth) components as well (see Chapter 2,~\cite{XENO1}). To approximate these smooth components in the solution at an exponential
rate, the mesh degree combination used in Theorem~\ref{thm6.2} is not sufficient and needs to be replaced by increasing the polynomial order $W$ or
by dividing $\Omega_2$ into smaller elements so that the smooth components in the solution are approximated at a sufficiently fast rate. As such the
overall convergence rate will be exponential only if the smooth components are also approximated at an exponential rate~\cite{XENO1}. It is shown
in~\cite{XENO1} (Theorem 2.5.2) and in~\cite{SCHW1} (Theorem 5.2) that this is possible to obtain if $f$ is a polynomial.
 
Hence, exponential convergence is achieved uniformly in $\epsilon$ using a simple two-element boundary layer mesh. This mesh is easier to implement than
a general $hp$-mesh, and moreover, performing computational experiments with more elements, we can not obtain better convergence rates
(see~\cite{XENO2}). The mesh-degree combination used by us is similar to the optimal mesh-degree combination proposed in~\cite{SCHE,SCHW1,XENO1}.

Since, $W=O(\frac{1}{\epsilon})$ therefore let us take $W=\frac{\beta}{\epsilon}$ for some constant $\beta>0$. Since $0<\alpha<1$ therefore
we can take $\alpha=e^{-\lambda}$, with $\lambda>0$. Now compute
\[\alpha^{2W_0}=\left(e^{-\lambda}\right)^{2W_0}=e^{-2\lambda W_0}.\]
But $W_0=\frac{\beta}{\epsilon}+\frac{1}{2}$, therefore $\alpha^{2W_0}=e^{\frac{2\lambda \beta}{\epsilon}}e^{-\lambda}=O\left(e^{-\frac{\gamma}{\epsilon}}\right),\quad \gamma=2\lambda\beta$.
Now, (\ref{eq6.16}) takes the form
\begin{align*}
\|U(x,y)-z(x,y)\|_{2,\epsilon,\Omega}^2\leq C {\epsilon}e^{-\frac{\gamma}{\epsilon}}.
\end{align*}
This states that the error decays \emph{faster than any polynomial in $\epsilon$}.
\begin{rem}
The error estimate in Theorem~\ref{thm6.2} is obtained using polynomials of degree $O(1)$ in the element $\Omega_2$ and the estimate remains valid
if $\Omega_2$ is divided into further smaller elements and polynomials of degree $O(1)$ are used.
\end{rem}
\begin{rem}
It is shown in~\cite{SCHW1,XENO1} that the constant $\alpha$ appearing in the Theorem~\ref{thm6.2} is given by
\begin{align}\label{eq6.21}
\alpha=\min\left\{\frac{e}{2\epsilon W_0},\max\left(\frac{\kappa\epsilon}{4},e^{-(\kappa-\delta)}\right)\right\},
\end{align}
with $\delta>\frac{\log W}{2W}$ and $0<\kappa_0\leq\kappa<\frac{4}{e}$.
\end{rem}
\begin{rem}
It is also known~\cite{SCHW1,XENO1} that the constant $\alpha$ in Theorem~\ref{thm6.2} may be chosen so that $\kappa_0 e=\frac{4}{e^{\kappa_0}}$,
which gives $\kappa_0\approx0.71$. With this choice, $\alpha$ takes the value $\alpha\approx e^{-\kappa_0}$ in (\ref{eq6.21}), which simplifies the error bound in
(\ref{eq6.16}) when only two elements are used in the mesh. However, this value of $\kappa_0$ is not optimal as reflected in numerical results
(in Section~\ref{sec7}). We also refer to the section on numerical results in~\cite{SCHW1} for more details on optimal choice of $\kappa_0$.
\end{rem}

\section{Numerical results}\label{sec7}

The spectral element approximation resulting from the minimization problem is denoted by $u_{\text{sem}}$, and the exact solution of the problem is denoted by $u_{\text{ex}}$. The percentage relative error in the parameter-weighted $H^2_{\epsilon}$-norm (as defined in Eq.~\ref{eq2.5}) is given by

\[\|E\|_{\text{rel}} =\frac{\|u_{\text{sem}}-u_{\text{ex}}\|_{2,\epsilon,\Omega}}{\|u_{\text{ex}}\|_{2,\epsilon,\Omega}}\times 100 \%.\]

All the numerical simulations reported in this section were executed with codes developed in \texttt{FORTRAN-90} on a high-performance computing system consisting of six interconnected nodes equipped with Intel Ivy Bridge processors, amounting to 120 cores in total. The detailed configuration of the cluster is as follows: number of physical CPUs – 12, cores per CPU – 10, total CPU cores – 120, and RAM – 96~GB. In all examples, each element of the discretized domain is assigned to a unique processor. The Message Passing Interface (MPI) library is used for inter-processor communication. To demonstrate the exponential rate of convergence, the relative error is plotted on a semi-logarithmic scale. In all examples presented in this section, four different discretizations of the domain $\Omega = (-1,1)^2$ are used, as illustrated in Figure~\ref{fig:mesh-config}.

\begin{figure}[ht]
\centering
\subfigure[]{
\includegraphics[width=0.23\textwidth]{fig1.eps}
\label{fig:mesh-config-a}
}
\hspace{-0.25cm}
\subfigure[]{
\includegraphics[width=0.23\textwidth]{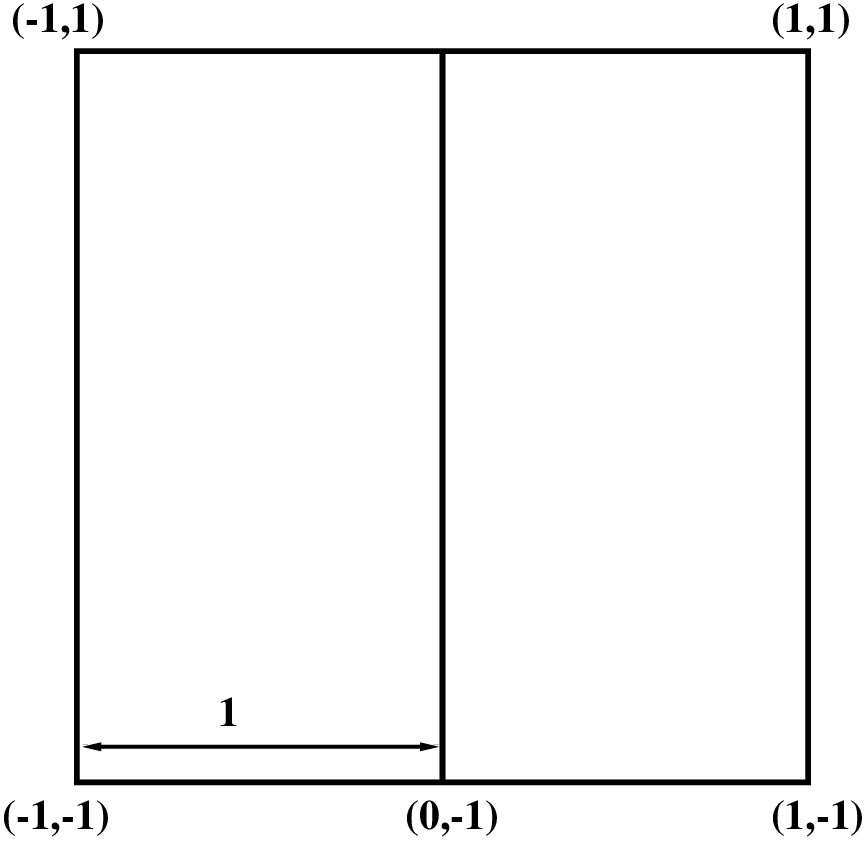}
\label{fig:mesh-config-b}
}
\hspace{-0.25cm}
\subfigure[]{
\includegraphics[width=0.23\textwidth]{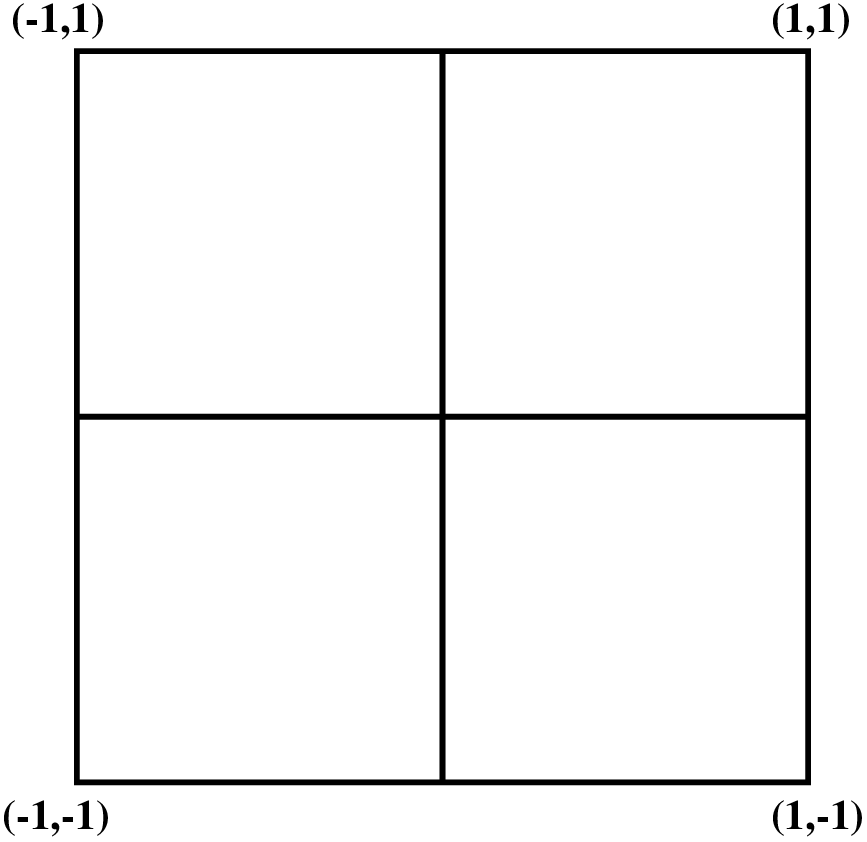}
\label{fig:mesh-config-c}
}
\hspace{-0.25cm}
\subfigure[]{
\includegraphics[width=0.23\textwidth]{fig4.eps}
\label{fig:mesh-config-d}
}
\caption[]{Decomposition of the square domain $\Omega = [-1,1]^2$ into \subref{fig:mesh-config-a} 1 element,
\subref{fig:mesh-config-b} 2 elements, \subref{fig:mesh-config-c} 4 elements and \subref{fig:mesh-config-d} 9 elements.}
\label{fig:mesh-config}
\end{figure}

We execute the Preconditioned Conjugate Gradient Method (PCGM) until the stopping criterion is satisfied, which requires that the relative norm of the residual vector associated with the normal equations is required to be smaller than a prescribed tolerance, denoted by $\text{tol}$. To meet this criterion, approximately 
$O\left(\frac{\sqrt{\kappa}}{2} \left|\log\left(\frac{2}{\text{tol}}\right)\right|\right)$ PCGM iterations are needed, where $\kappa$ denotes the condition number of the preconditioned system. This implies that achieving an approximate solution with accuracy $O\left(\frac{\sqrt{\log W}}{W}\right)$ requires $O\left((\log W)^2\right)$ iterations. Each iteration of the PCGM incurs a computational cost of $O(W^2)$ when executed on a parallel computer equipped with $O(W)$ processors.
 
Therefore, the total computational time required to achieve an overall solution accuracy of $O\left(\frac{\log W}{W^2}\right)$ is $O\left(W^2(\log W)^2\right)$ on such a parallel system.

\begin{example}[\textbf{Boundary layers on the left and bottom edges with a corner layer}]\label{examp1}
\end{example}

Considers the boundary layer problem~(\ref{eq2.6})--(\ref{eq2.7}) on \( \Omega = (0,1)^2 \) with \( a = 0 \), where the source term \( f(x,y) \) is chosen such that the exact solution is
\begin{align*}
w_\epsilon(x,y) = x^3(1 + y^2) + \sin(\pi x^2) + \cos\left(\frac{\pi y}{2}\right) + (1 + x + y) \left( e^{-2x/\epsilon} + e^{-2y/\epsilon} \right).
\end{align*}
This solution exhibits boundary layers along the left (\( x=0 \)) and bottom (\( y=0 \)) edges with corner layer at the origin \( (0,0) \)~\cite{NMM} (see Figure~\ref{Soln_Ex1}) .

\begin{figure}[h!]
\centering
\includegraphics[width=1.0\textwidth]{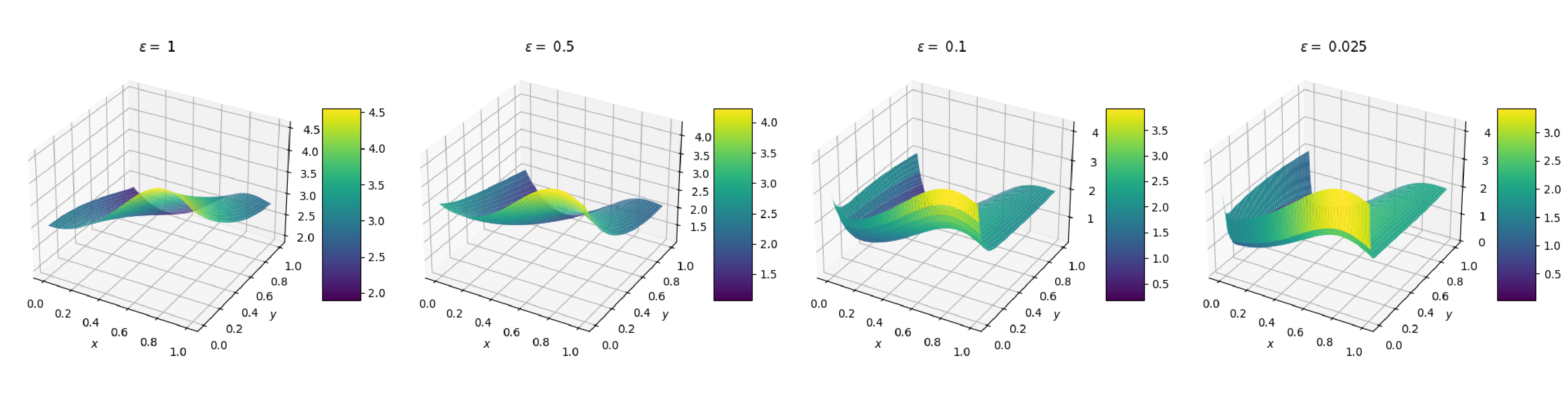}
\caption{Exact solution (Example~\ref{examp1}).}
\label{Soln_Ex1}
\end{figure}

We apply $p$-version of the method on uniform meshes with 1, 2, 4, and 9 elements, and investigate the convergence behavior on a semi $\log$-scale for various boundary layer parameters \( \epsilon \in \{1, 0.5, 0.1, 0.075, 0.025, 0.01\} \) with polynomial degrees \( W = 2, 4, \ldots, 32 \). The results are presented in Tables~\ref{ex1_ele1}--\ref{ex1_ele9} while the resulting error curves are plotted in Figure~\ref{fig:ex1}.

From Figure~\ref{fig:ex1}, we observe that exponential convergence is achieved for larger values of \( \epsilon \), especially on finer meshes. However, as \( \epsilon \) decreases, the presence of sharp boundary layer makes the problem increasingly challenging. On a single-element mesh (Figure~\ref{fig:ex1}(a)), the method fails to resolve the layers for small \( \epsilon \), resulting in saturation or even increase in error. With two and four elements (Figures~\ref{fig:ex1}(b)--(c)), performance improves moderately, although the resolution remains inadequate for capturing sharp boundary layers. On a 9-element mesh (Figure~\ref{fig:ex1}(d)), significant improvement is seen, particularly for moderate \( \epsilon \). For \( \epsilon \leq 0.025 \), however, the convergence rate remains suboptimal, indicating the necessity of mesh adaptation or $rp$-refinement for accurately capturing the layer behavior.

\begin{table}[ht!]
\centering
\caption{Relative error for various scales of $\epsilon$ for Example~\ref{examp1} (on a single element)}
\label{ex1_ele1}
\begin{tabular}{c | c c c c c c}
\toprule
\textbf{ $W$} & \boldmath{$\epsilon = 1$} & \boldmath{$\epsilon = 0.5$} & \boldmath{$\epsilon = 0.1$} & \boldmath{$\epsilon = 0.075$} & \boldmath{$\epsilon = 0.025$} & \boldmath{$\epsilon = 0.01$} \\
\midrule
2  & $1.68\mathrm{E}{+}02$ & $8.81\mathrm{E}{+}00$ & $4.68\mathrm{E}{+}01$ & $4.84\mathrm{E}{+}01$ & $5.06\mathrm{E}{+}01$ & $5.08\mathrm{E}{+}01$ \\
4  & $4.26\mathrm{E}{+}01$ & $1.88\mathrm{E}{+}01$ & $4.95\mathrm{E}{+}01$ & $7.00\mathrm{E}{+}01$ & $2.37\mathrm{E}{+}02$ & $8.08\mathrm{E}{+}01$ \\
8  & $1.75\mathrm{E}{+}00$ & $7.18\mathrm{E}{-}01$ & $6.53\mathrm{E}{+}00$ & $1.26\mathrm{E}{+}01$ & $6.03\mathrm{E}{+}01$ & $2.22\mathrm{E}{+}02$ \\
16 & $1.10\mathrm{E}{-}05$ & $5.81\mathrm{E}{-}06$ & $3.10\mathrm{E}{-}03$ & $3.23\mathrm{E}{-}02$ & $5.50\mathrm{E}{+}00$ & $3.22\mathrm{E}{+}01$ \\
24 & $9.98\mathrm{E}{-}09$ & $5.24\mathrm{E}{-}09$ & $4.72\mathrm{E}{-}07$ & $2.79\mathrm{E}{-}06$ & $1.62\mathrm{E}{-}01$ & $5.91\mathrm{E}{+}00$ \\
32 & $4.61\mathrm{E}{-}08$ & $2.32\mathrm{E}{-}08$ & $1.34\mathrm{E}{-}06$ & $1.85\mathrm{E}{-}06$ & $1.18\mathrm{E}{-}03$ & $7.68\mathrm{E}{-}01$ \\
\bottomrule
\end{tabular}
\end{table}
\begin{table}[ht!]
\centering
\caption{Relative error for various scales of $\epsilon$ for Example~\ref{examp1} (on 2 elements)}
\label{ex1_ele2}
\begin{tabular}{c | c c c c c c}
\toprule
\textbf{ $W$} & \boldmath{$\epsilon = 1$} & \boldmath{$\epsilon = 0.5$} & \boldmath{$\epsilon = 0.1$} & \boldmath{$\epsilon = 0.075$} & \boldmath{$\epsilon = 0.025$} & \boldmath{$\epsilon = 0.01$} \\
\midrule
2  & $2.16\mathrm{E}{+}01$ & $9.46\mathrm{E}{+}00$ & $5.89\mathrm{E}{+}01$ & $6.81\mathrm{E}{+}01$ & $5.85\mathrm{E}{+}01$ & $5.86\mathrm{E}{+}01$ \\
4  & $8.42\mathrm{E}{+}00$ & $4.36\mathrm{E}{+}00$ & $3.73\mathrm{E}{+}01$ & $5.45\mathrm{E}{+}01$ & $1.29\mathrm{E}{+}02$ & $1.29\mathrm{E}{+}02$ \\
8  & $8.78\mathrm{E}{-}03$ & $5.13\mathrm{E}{-}03$ & $4.51\mathrm{E}{+}00$ & $8.73\mathrm{E}{+}00$ & $4.57\mathrm{E}{+}01$ & $1.82\mathrm{E}{+}02$ \\
16 & $6.21\mathrm{E}{-}08$ & $5.28\mathrm{E}{-}08$ & $2.14\mathrm{E}{-}03$ & $2.25\mathrm{E}{-}02$ & $3.91\mathrm{E}{+}00$ & $2.39\mathrm{E}{+}01$ \\
24 & $2.58\mathrm{E}{-}08$ & $2.35\mathrm{E}{-}08$ & $4.21\mathrm{E}{-}06$ & $3.37\mathrm{E}{-}06$ & $1.15\mathrm{E}{-}01$ & $4.28\mathrm{E}{+}00$ \\
32 & $1.34\mathrm{E}{-}07$ & $6.50\mathrm{E}{-}08$ & $7.96\mathrm{E}{-}06$ & $1.60\mathrm{E}{-}05$ & $6.57\mathrm{E}{-}03$ & $5.58\mathrm{E}{-}01$  \\
\bottomrule
\end{tabular}
\end{table}
\begin{table}[ht!]
\centering
\caption{Relative error for various scales of $\epsilon$ for Example~\ref{examp1} (on 4 elements)}
\label{ex1_ele4}
\begin{tabular}{c | c c c c c c}
\toprule
\textbf{ $W$} & \boldmath{$\epsilon = 1$} & \boldmath{$\epsilon = 0.5$} & \boldmath{$\epsilon = 0.1$} & \boldmath{$\epsilon = 0.075$} & \boldmath{$\epsilon = 0.025$} & \boldmath{$\epsilon = 0.01$} \\
\midrule
2  & $2.04\mathrm{E}{+}01$ & $9.50\mathrm{E}{+}00$ & $7.25\mathrm{E}{+}01$ & $7.90\mathrm{E}{+}01$ & $7.15\mathrm{E}{+}01$ & $7.13\mathrm{E}{+}01$ \\
4  & $9.83\mathrm{E}{+}00$ & $4.24\mathrm{E}{+}00$ & $1.87\mathrm{E}{+}01$ & $2.73\mathrm{E}{+}01$ & $1.19\mathrm{E}{+}02$ & $8.91\mathrm{E}{+}01$ \\
8  & $9.81\mathrm{E}{-}03$ & $4.59\mathrm{E}{-}03$ & $5.47\mathrm{E}{-}01$ & $1.62\mathrm{E}{+}00$ & $1.87\mathrm{E}{+}01$ & $1.13\mathrm{E}{+}02$ \\
16 & $7.84\mathrm{E}{-}08$ & $4.48\mathrm{E}{-}08$ & $2.66\mathrm{E}{-}06$ & $6.00\mathrm{E}{-}05$ & $3.04\mathrm{E}{-}01$ & $6.28\mathrm{E}{+}00$ \\
24 & $3.76\mathrm{E}{-}08$ & $2.83\mathrm{E}{-}08$ & $2.02\mathrm{E}{-}06$ & $1.55\mathrm{E}{-}05$ & $2.01\mathrm{E}{-}03$ & $3.83\mathrm{E}{-}01$ \\
32 & $1.85\mathrm{E}{-}07$ & $9.34\mathrm{E}{-}08$ & $5.89\mathrm{E}{-}06$ & $4.18\mathrm{E}{-}05$ & $1.62\mathrm{E}{-}03$ & $1.02\mathrm{E}{-}01$ \\
\bottomrule
\end{tabular}
\end{table}
\begin{table}[ht!]
\centering
\caption{Relative error for various scales of $\epsilon$ for Example~\ref{examp1} (on 9 elements)}
\label{ex1_ele9}
\begin{tabular}{c | c c c c c c}
\toprule
\textbf{ $W$} & 
\boldmath{$\epsilon = 1$} & \boldmath{$\epsilon = 0.5$} & \boldmath{$\epsilon = 0.1$} & \boldmath{$\epsilon = 0.075$} & \boldmath{$\epsilon = 0.025$} & \boldmath{$\epsilon = 0.01$} \\
\midrule
2  & $1.74\mathrm{E}{+}01$ & $9.08\mathrm{E}{+}00$ & $6.74\mathrm{E}{+}01$ & $7.27\mathrm{E}{+}01$ & $7.89\mathrm{E}{+}01$ & $7.85\mathrm{E}{+}01$ \\
4  & $2.56\mathrm{E}{+}00$ & $9.93\mathrm{E}{-}01$ & $8.42\mathrm{E}{+}00$ & $1.38\mathrm{E}{+}01$ & $1.04\mathrm{E}{+}02$ & $9.20\mathrm{E}{+}01$ \\
8  & $7.66\mathrm{E}{-}04$ & $3.68\mathrm{E}{-}04$ & $7.46\mathrm{E}{-}02$ & $2.82\mathrm{E}{-}01$ & $9.15\mathrm{E}{+}00$ & $9.68\mathrm{E}{+}01$ \\
16 & $1.04\mathrm{E}{-}07$ & $6.72\mathrm{E}{-}08$ & $3.87\mathrm{E}{-}06$ & $8.32\mathrm{E}{-}06$ & $2.05\mathrm{E}{-}02$ & $1.68\mathrm{E}{+}00$ \\
24 & $8.40\mathrm{E}{-}08$ & $4.69\mathrm{E}{-}08$ & $7.70\mathrm{E}{-}06$ & $1.53\mathrm{E}{-}05$ & $7.78\mathrm{E}{-}04$ & $6.30\mathrm{E}{-}02$ \\
32 & $1.85\mathrm{E}{-}07$ & $2.09\mathrm{E}{-}07$ & $5.66\mathrm{E}{-}05$ & $2.56\mathrm{E}{-}05$ & $6.73\mathrm{E}{-}03$ & $2.37\mathrm{E}{-}01$  \\
\bottomrule
\end{tabular}
\end{table}

These results confirm that although the $p$-version achieves exponential convergence for smooth solutions, layer problems demand mesh refinement or layer-adapted techniques to accurately capture sharp gradients.

\begin{figure}[ht!]
\centering
\subfigure[]{
\includegraphics[width=0.50\textwidth]{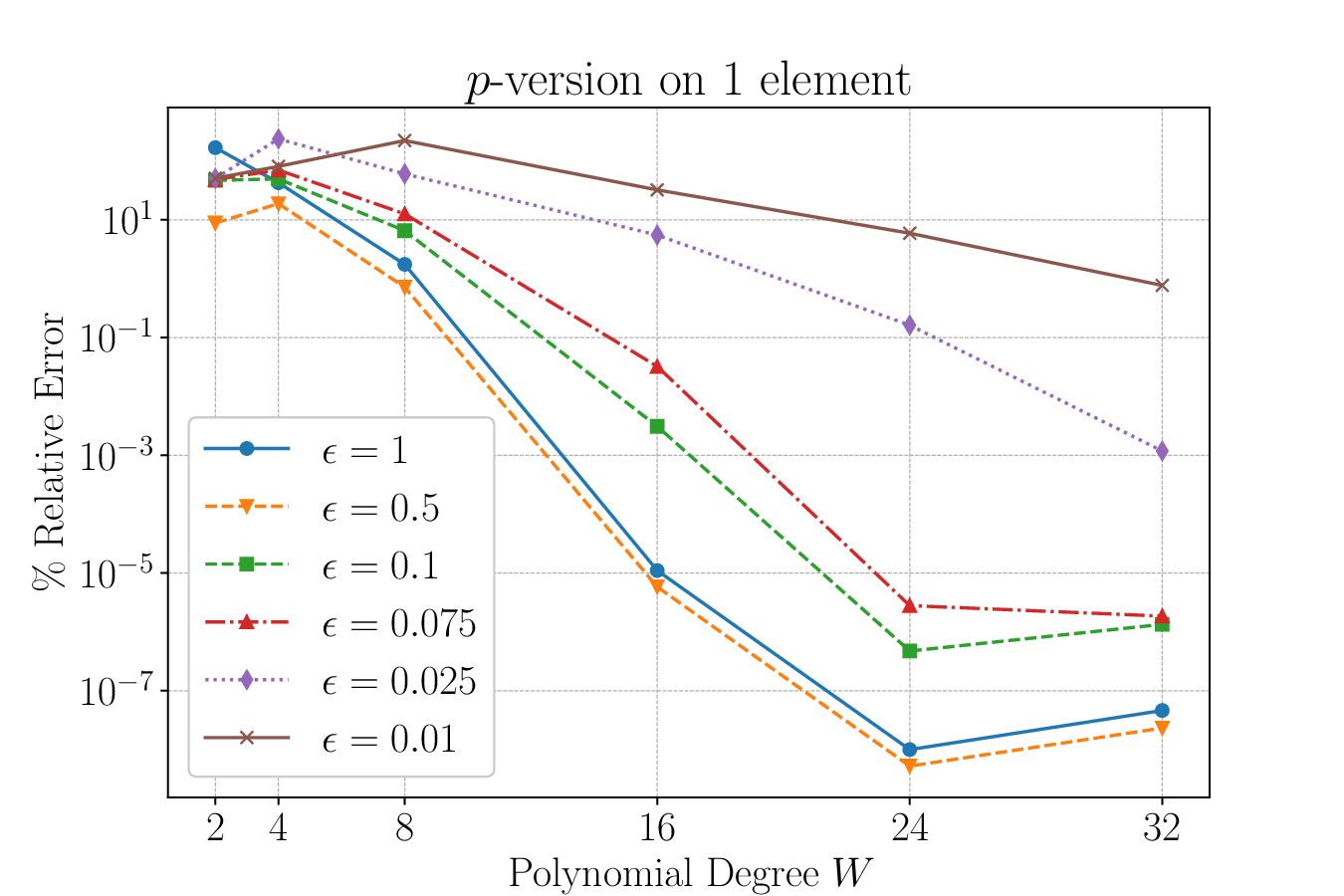}
\label{fig:ex1-a}
}
\hspace{-1.3cm}
\subfigure[]{
\includegraphics[width=0.50\textwidth]{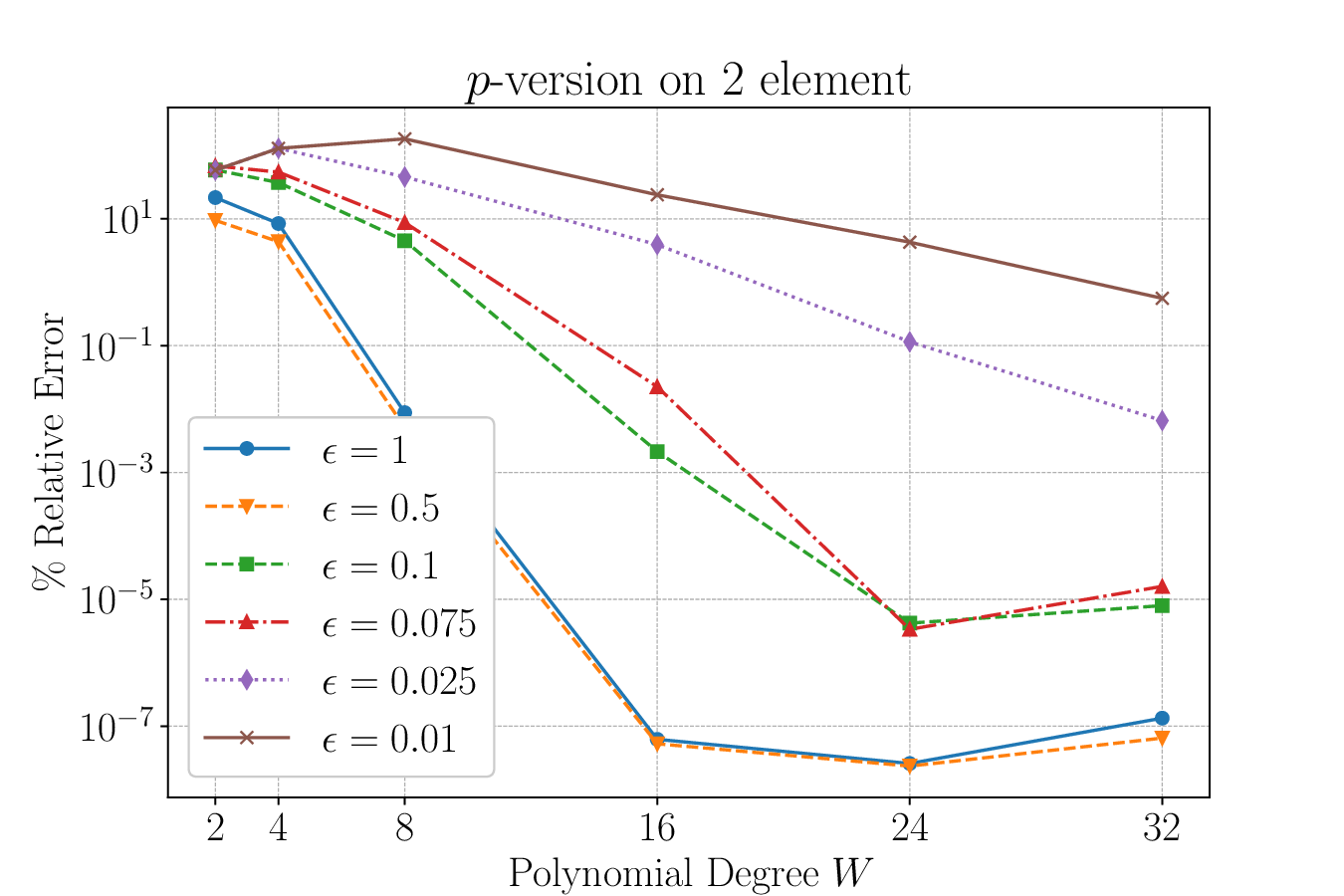}
\label{fig:ex1-b}
}
\hspace{0.5cm}
\subfigure[]{
\includegraphics[width=0.50\textwidth]{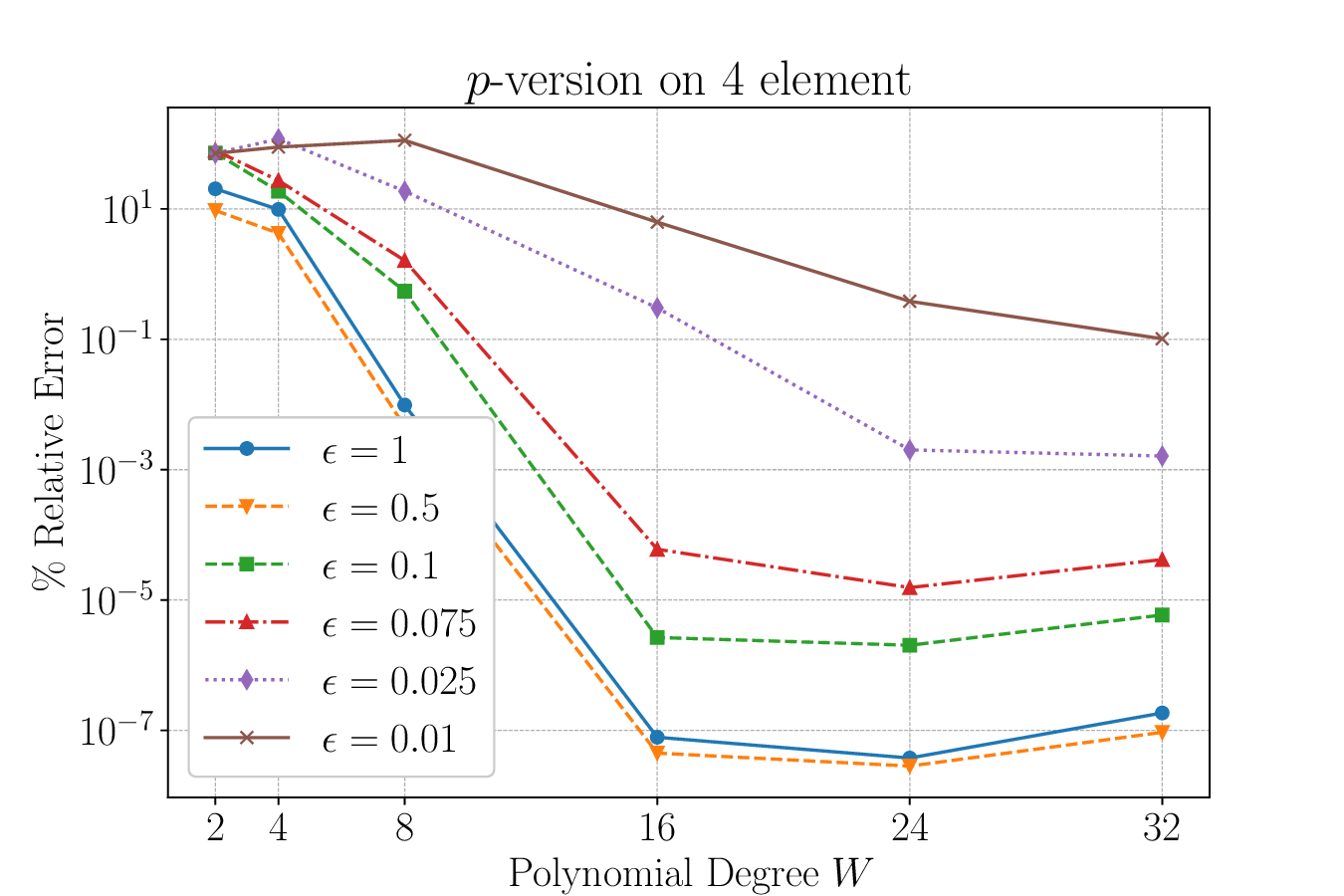}
\label{fig:ex1-c}
}
\hspace{-1.3cm}
\subfigure[]{
\includegraphics[width=0.50\textwidth]{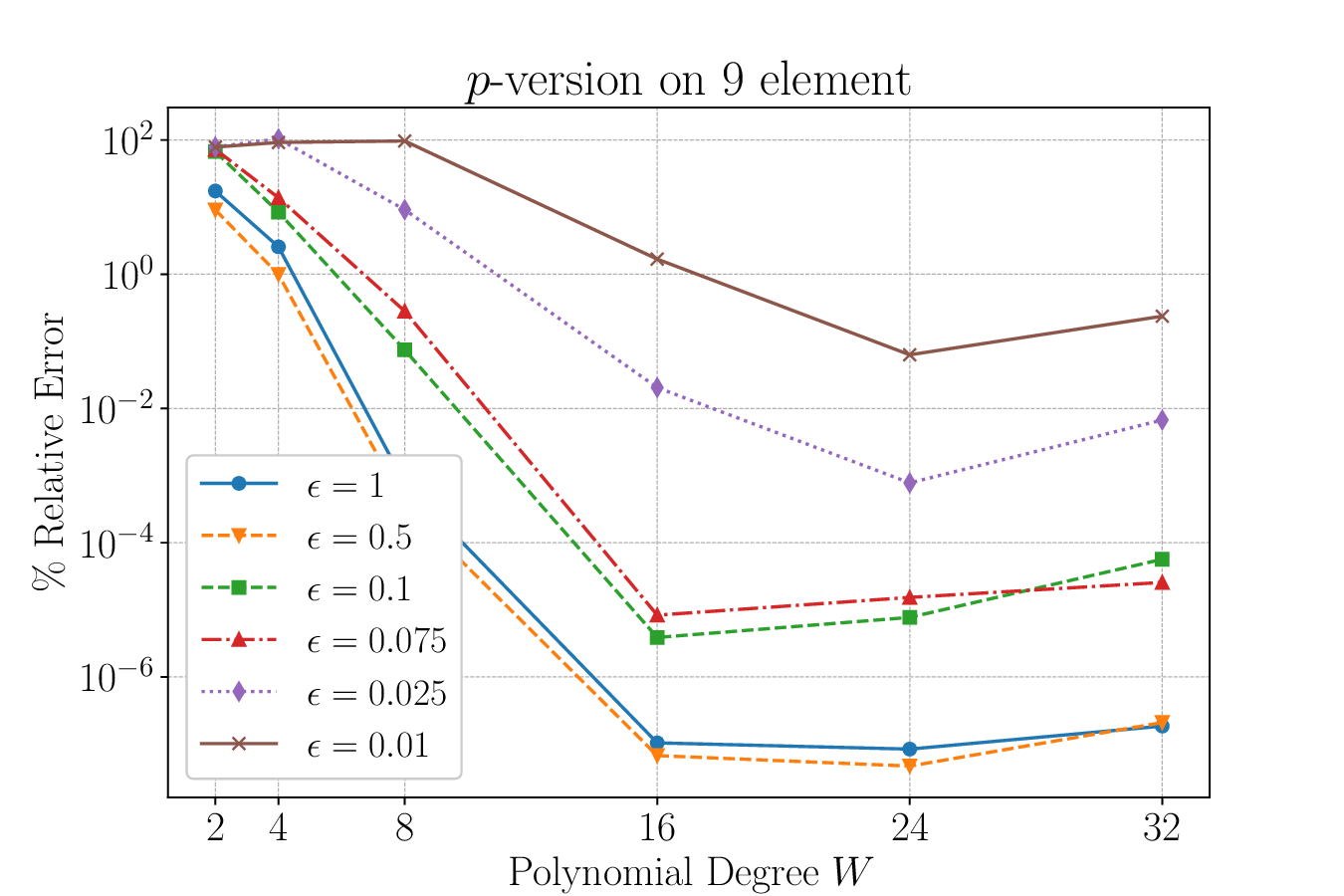}
\label{fig:ex1-d}
}  
\caption{Performance of the $p$-version on different mesh configurations (Example~\ref{examp1}).}
\label{fig:ex1}
\end{figure}
\newpage
\begin{example}[\textbf{Boundary layers on the top and right edges}] \label{example2}
\end{example}

We now consider the boundary layer problem~(\ref{eq2.6})--(\ref{eq2.7}) on the unit square $\Omega = (0,1)^2$ with $a = 0$. The forcing function $f(x,y)$ is chosen such that the exact solution is
\begin{align*}
w_\epsilon(x,y) = \left( \sin\left(\frac{\pi x}{2}\right) - \frac{e^{-(1-x)/\epsilon} - e^{-1/\epsilon}}{1 - e^{-1/\epsilon}} \right)
\left( y - \frac{e^{-2(1-y)/\epsilon} - e^{-2/\epsilon}}{1 - e^{-2/\epsilon}} \right),
\end{align*}
which exhibits exponential boundary layers along the right ($x=1$) and top ($y=1$) edges that become sharper as $\epsilon$ decreases~\cite{xeno2016}. Figure~\ref{Soln_Ex2} shows the exact solution with steep gradients near the layer regions.

\begin{figure}[h!]
\centering
\includegraphics[width=1.0\textwidth]{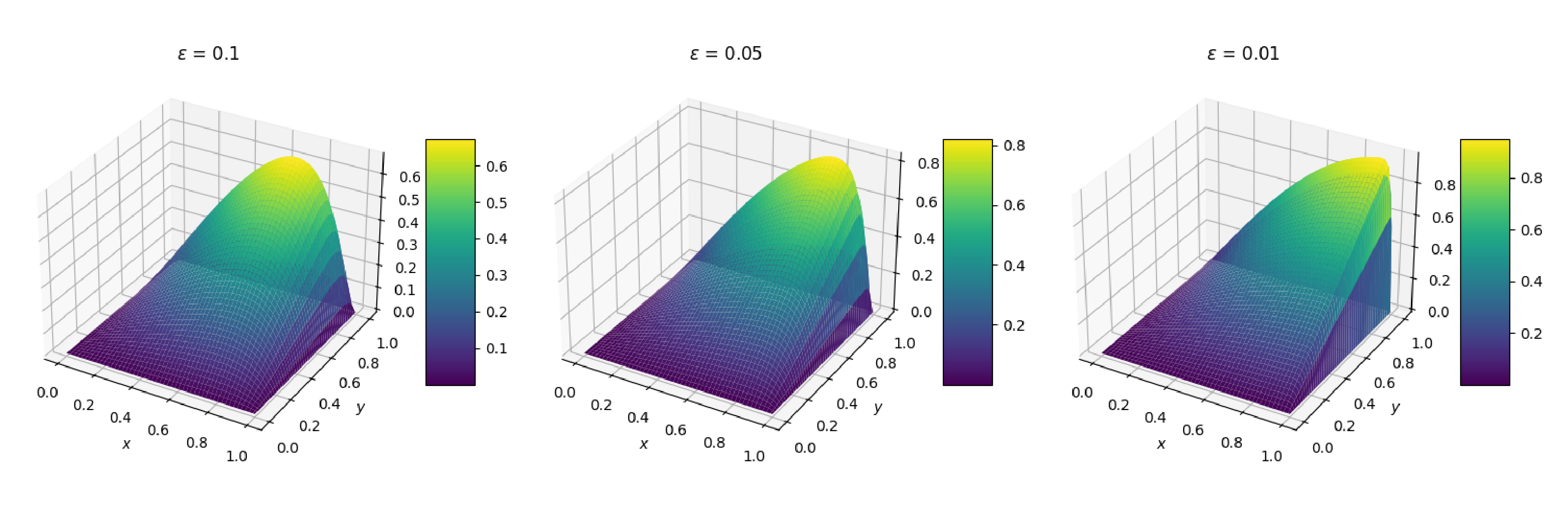}
\caption{Exact solution (Example~\ref{example2}).}
\label{Soln_Ex2}
\end{figure}

Tables~\ref{ex2_ele1}--\ref{ex2_ele9} report the relative $H_{\epsilon}^2$-errors in semi $\log$-scale for uniform meshes with 1, 2, 4, and 9 spectral elements with polynomial degrees \( W = 2, 4, \ldots, 32 \), and Figure~\ref{fig:ex2} illustrates the convergence behavior for various $\epsilon$.

\begin{table}[h!]
\centering
\caption{Relative error for different scales of $\epsilon$ for Example~\ref{example2} (on a single element)}
\label{ex2_ele1}
\begin{tabular}{c | c c c c c c}
\toprule
\textbf{$W$} & \boldmath{$\epsilon = 1$} & \boldmath{$\epsilon = 0.5$} & \boldmath{$\epsilon = 0.1$} & \boldmath{$\epsilon = 0.075$} & \boldmath{$\epsilon = 0.025$} & \boldmath{$\epsilon = 0.01$} \\
\midrule
2  & $3.34\mathrm{E}{+}01$ & $1.41\mathrm{E}{+}01$ & $1.00\mathrm{E}{+}02$ & $1.00\mathrm{E}{+}02$ & $1.00\mathrm{E}{+}02$ & $1.00\mathrm{E}{+}02$ \\
4  & $8.85\mathrm{E}{-}01$ & $4.08\mathrm{E}{+}00$ & $8.24\mathrm{E}{+}01$ & $1.23\mathrm{E}{+}02$ & $1.00\mathrm{E}{+}02$ & $1.00\mathrm{E}{+}02$ \\
8  & $3.54\mathrm{E}{-}04$ & $8.28\mathrm{E}{-}03$ & $9.02\mathrm{E}{+}00$ & $1.71\mathrm{E}{+}01$ & $1.07\mathrm{E}{+}02$ & $2.68\mathrm{E}{+}02$ \\
16 & $7.26\mathrm{E}{-}07$ & $9.08\mathrm{E}{-}07$ & $5.29\mathrm{E}{-}03$ & $4.85\mathrm{E}{-}02$ & $7.19\mathrm{E}{+}00$ & $4.27\mathrm{E}{+}01$ \\
24 & $2.11\mathrm{E}{-}10$ & $1.75\mathrm{E}{-}10$ & $7.80\mathrm{E}{-}08$ & $5.10\mathrm{E}{-}06$ & $2.15\mathrm{E}{-}01$ & $7.68\mathrm{E}{+}00$ \\
32 & $2.12\mathrm{E}{-}10$ & $1.50\mathrm{E}{-}10$ & $3.01\mathrm{E}{-}10$ & $1.32\mathrm{E}{-}09$ & $1.50\mathrm{E}{-}03$ & $9.97\mathrm{E}{-}01$ \\
\bottomrule
\end{tabular}
\end{table}
\begin{table}[h!]
\centering
\caption{Relative error for different scales of $\epsilon$ for Example~\ref{example2} (on 2 elements)}
\label{ex2_ele2}
\begin{tabular}{c | c c c c c c}
\toprule
\textbf{ $W$} & \boldmath{$\epsilon = 1$} & \boldmath{$\epsilon = 0.5$} & \boldmath{$\epsilon = 0.1$} & \boldmath{$\epsilon = 0.075$} & \boldmath{$\epsilon = 0.025$} & \boldmath{$\epsilon = 0.01$} \\
\midrule
2  & $1.51\mathrm{E}{+}01$ & $1.23\mathrm{E}{+}01$ & $1.00\mathrm{E}{+}02$ & $1.00\mathrm{E}{+}02$ & $1.00\mathrm{E}{+}02$ & $1.00\mathrm{E}{+}02$ \\
4  & $7.50\mathrm{E}{-}01$ & $4.16\mathrm{E}{+}00$ & $7.97\mathrm{E}{+}01$ & $1.22\mathrm{E}{+}02$ & $1.00\mathrm{E}{+}02$ & $1.00\mathrm{E}{+}02$ \\
8  & $2.25\mathrm{E}{-}04$ & $8.35\mathrm{E}{-}03$ & $9.09\mathrm{E}{+}00$ & $1.73\mathrm{E}{+}01$ & $9.93\mathrm{E}{+}01$ & $2.58\mathrm{E}{+}02$ \\
16 & $4.80\mathrm{E}{-}07$ & $4.64\mathrm{E}{-}07$ & $5.33\mathrm{E}{-}03$ & $4.88\mathrm{E}{-}02$ & $7.22\mathrm{E}{+}00$ & $4.26\mathrm{E}{+}01$ \\
24 & $3.61\mathrm{E}{-}09$ & $2.89\mathrm{E}{-}09$ & $7.84\mathrm{E}{-}08$ & $5.13\mathrm{E}{-}06$ & $2.16\mathrm{E}{-}01$ & $7.69\mathrm{E}{+}00$ \\
32 & $1.93\mathrm{E}{-}08$ & $1.44\mathrm{E}{-}08$ & $8.30\mathrm{E}{-}09$ & $3.69\mathrm{E}{-}08$ & $1.51\mathrm{E}{-}03$ & $9.98\mathrm{E}{-}01$ \\
\bottomrule
\end{tabular}
\end{table}
\begin{table}[h!]
\centering
\caption{Relative error for different scales of $\epsilon$ for Example~\ref{example2} (on 4 elements)}
\label{ex2_ele4}
\begin{tabular}{c | c c c c c c}
\toprule
\textbf{ $W$} & \boldmath{$\epsilon = 1$} & \boldmath{$\epsilon = 0.5$} & \boldmath{$\epsilon = 0.1$} & \boldmath{$\epsilon = 0.075$} & \boldmath{$\epsilon = 0.025$} & \boldmath{$\epsilon = 0.01$} \\
\midrule
2  & $5.95\mathrm{E}{+}00$ & $8.69\mathrm{E}{+}00$ & $7.90\mathrm{E}{+}01$ & $1.00\mathrm{E}{+}02$ & $1.00\mathrm{E}{+}02$ & $1.00\mathrm{E}{+}02$ \\
4  & $1.26\mathrm{E}{-}01$ & $7.17\mathrm{E}{-}01$ & $5.62\mathrm{E}{+}01$ & $7.86\mathrm{E}{+}01$ & $1.39\mathrm{E}{+}02$ & $1.00\mathrm{E}{+}02$ \\
8  & $2.12\mathrm{E}{-}04$ & $4.22\mathrm{E}{-}04$ & $7.11\mathrm{E}{-}01$ & $2.01\mathrm{E}{+}00$ & $7.41\mathrm{E}{+}01$ & $1.18\mathrm{E}{+}02$ \\
16 & $5.33\mathrm{E}{-}07$ & $3.90\mathrm{E}{-}07$ & $6.78\mathrm{E}{-}06$ & $9.77\mathrm{E}{-}05$ & $3.55\mathrm{E}{-}01$ & $7.06\mathrm{E}{+}00$ \\
24 & $5.23\mathrm{E}{-}09$ & $4.01\mathrm{E}{-}09$ & $1.12\mathrm{E}{-}08$ & $2.26\mathrm{E}{-}08$ & $4.49\mathrm{E}{-}04$ & $4.20\mathrm{E}{-}01$ \\
32 & $2.71\mathrm{E}{-}08$ & $2.15\mathrm{E}{-}08$ & $7.31\mathrm{E}{-}09$ & $1.73\mathrm{E}{-}08$ & $3.38\mathrm{E}{-}06$ & $7.62\mathrm{E}{-}03$ \\
\bottomrule
\end{tabular}
\end{table}

\begin{table}[h!]
\centering
\caption{Relative error for different scales of $\epsilon$ for Example~\ref{example2} (on 9 elements)}
\label{ex2_ele9}
\begin{tabular}{c | c c c c c c}
\toprule
\textbf{ $W$} & \boldmath{$\epsilon = 1$} & \boldmath{$\epsilon = 0.5$} & \boldmath{$\epsilon = 0.1$} & \boldmath{$\epsilon = 0.075$} & \boldmath{$\epsilon = 0.025$} & \boldmath{$\epsilon = 0.01$} \\
\midrule
2  & $2.71\mathrm{E}{+}00$ & $4.28\mathrm{E}{+}00$ & $7.32\mathrm{E}{+}01$ & $1.00\mathrm{E}{+}02$ & $1.00\mathrm{E}{+}02$ & $1.00\mathrm{E}{+}02$ \\
4  & $3.69\mathrm{E}{-}02$ & $2.32\mathrm{E}{-}01$ & $2.49\mathrm{E}{+}01$ & $6.63\mathrm{E}{+}01$ & $1.08\mathrm{E}{+}02$ & $1.00\mathrm{E}{+}02$ \\
8  & $2.54\mathrm{E}{-}04$ & $2.15\mathrm{E}{-}04$ & $9.93\mathrm{E}{-}02$ & $3.68\mathrm{E}{-}01$ & $6.59\mathrm{E}{+}01$ & $1.01\mathrm{E}{+}02$ \\
16 & $2.34\mathrm{E}{-}07$ & $2.04\mathrm{E}{-}07$ & $3.79\mathrm{E}{-}06$ & $3.35\mathrm{E}{-}06$ & $2.40\mathrm{E}{-}02$ & $1.85\mathrm{E}{+}00$ \\
24 & $1.20\mathrm{E}{-}08$ & $9.64\mathrm{E}{-}09$ & $8.69\mathrm{E}{-}09$ & $6.44\mathrm{E}{-}08$ & $5.51\mathrm{E}{-}06$ & $2.79\mathrm{E}{-}02$ \\
32 & $2.71\mathrm{E}{-}08$ & $4.98\mathrm{E}{-}08$ & $4.56\mathrm{E}{-}08$ & $6.44\mathrm{E}{-}08$ & $2.59\mathrm{E}{-}06$ & $1.72\mathrm{E}{-}04$ \\
\bottomrule
\end{tabular}
\end{table}

On the single-element mesh (Table~\ref{ex2_ele1}, Figure~\ref{fig:ex2-a}), the method shows rapid convergence for large $\epsilon$, but stagnates for smaller values due to insufficient resolution of boundary layers. With two elements (Table~\ref{ex2_ele2}, Figure~\ref{fig:ex2-b}), the performance improves slightly, though it remains inadequate for strong layers. Using four elements (Table~\ref{ex2_ele4}, Figure~\ref{fig:ex2-c}) yields significant gains: exponential convergence is observed for moderate and small $\epsilon$, with steadily decreasing errors as $W$ increases. On the 9-element mesh (Table~\ref{ex2_ele9}, Figure~\ref{fig:ex2-d}), robust exponential convergence is achieved even for $\epsilon = 0.01$, confirming that modest mesh refinement combined with $p$-refinement effectively resolves sharp boundary layers, which clearly illustrates the limitations of pure $p$-refinement for boundary layer problems and the benefit of combining it with minimal $h$-refinement to efficiently capture boundary layer behavior.

\begin{figure}[h!]
\centering
\subfigure[]{
\includegraphics[width=0.50\textwidth]{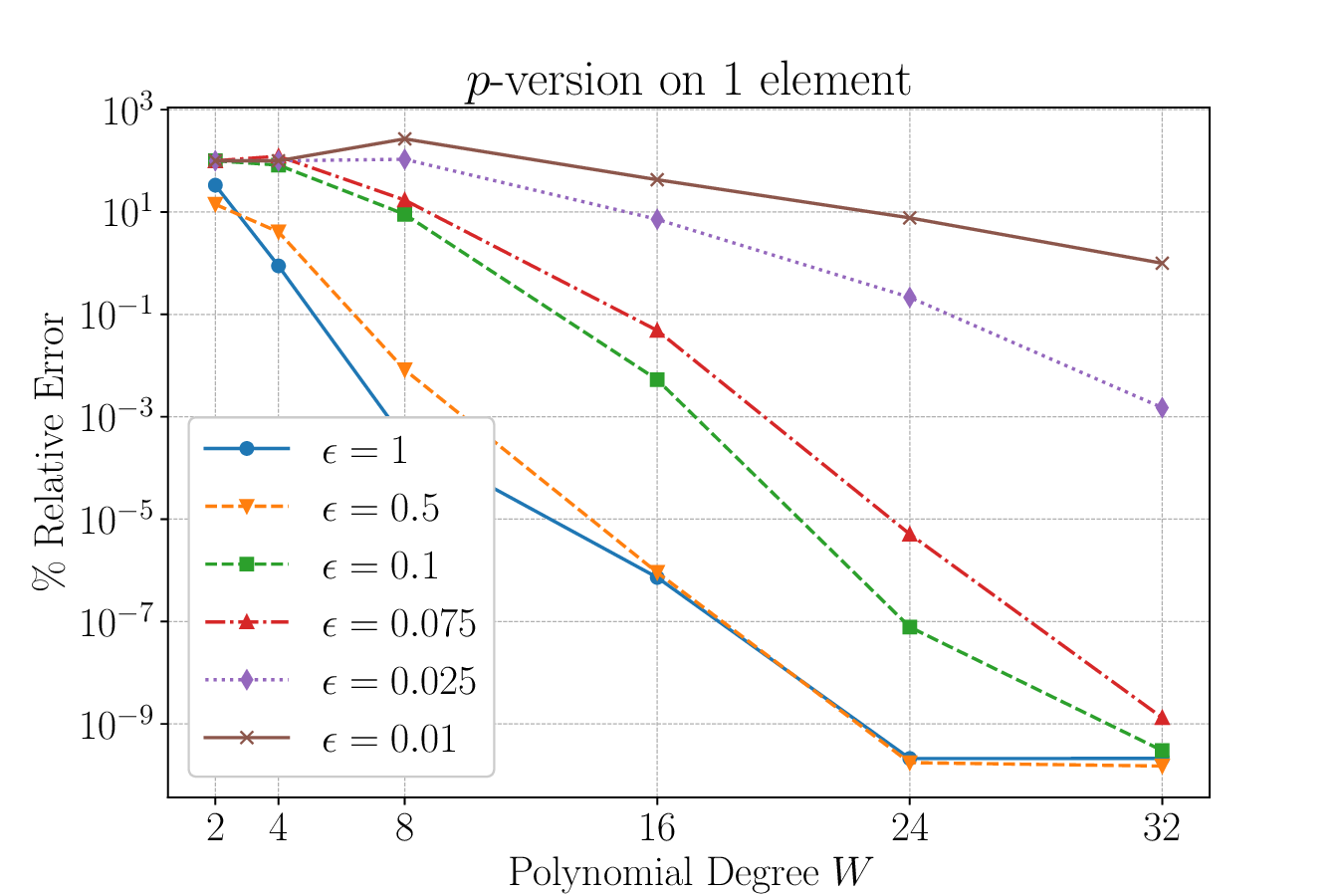}
\label{fig:ex2-a}
}
\hspace{-1.3cm}
\subfigure[]{
\includegraphics[width=0.50\textwidth]{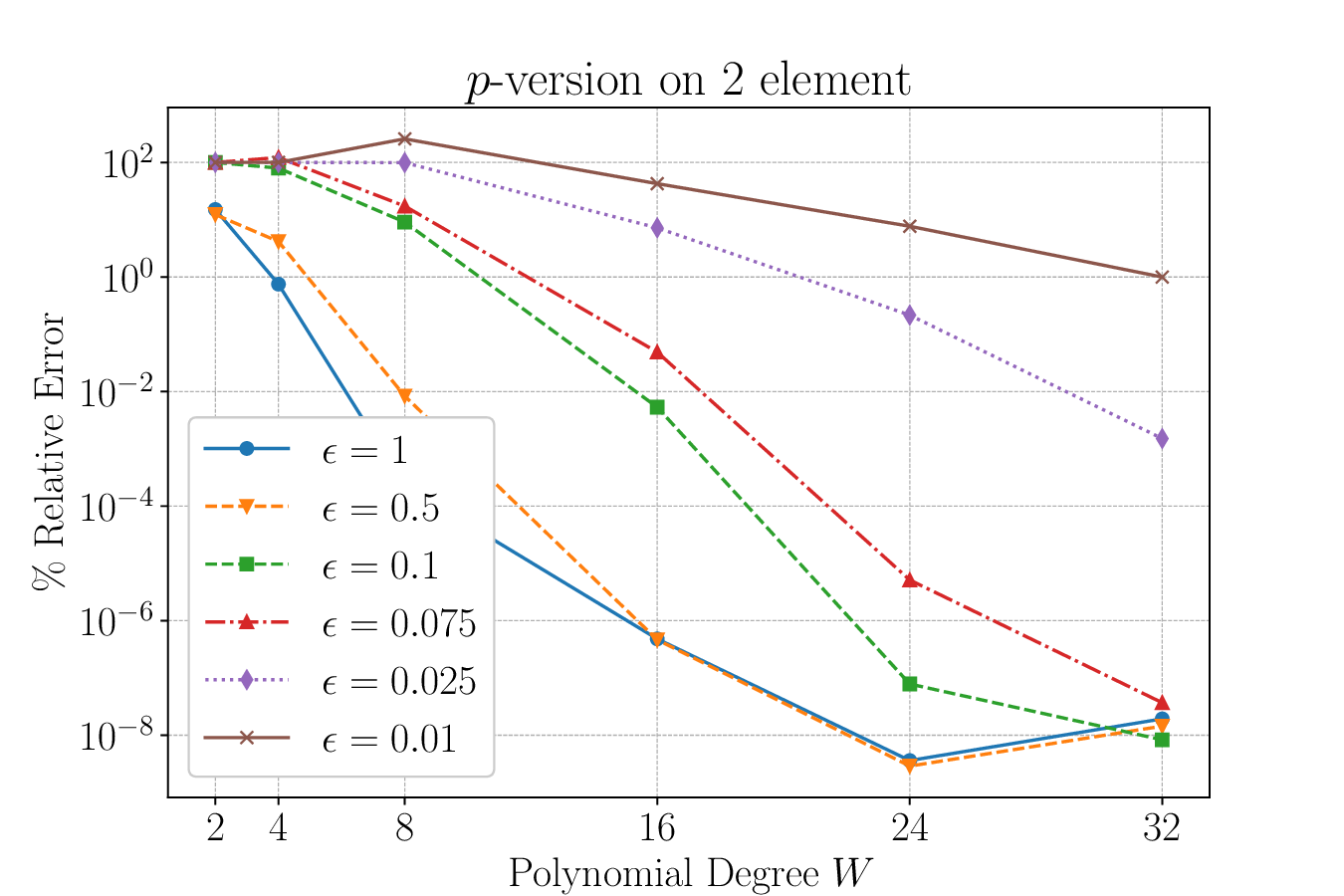}
\label{fig:ex2-b}
}
\hspace{0.5cm}
\subfigure[]{
\includegraphics[width=0.50\textwidth]{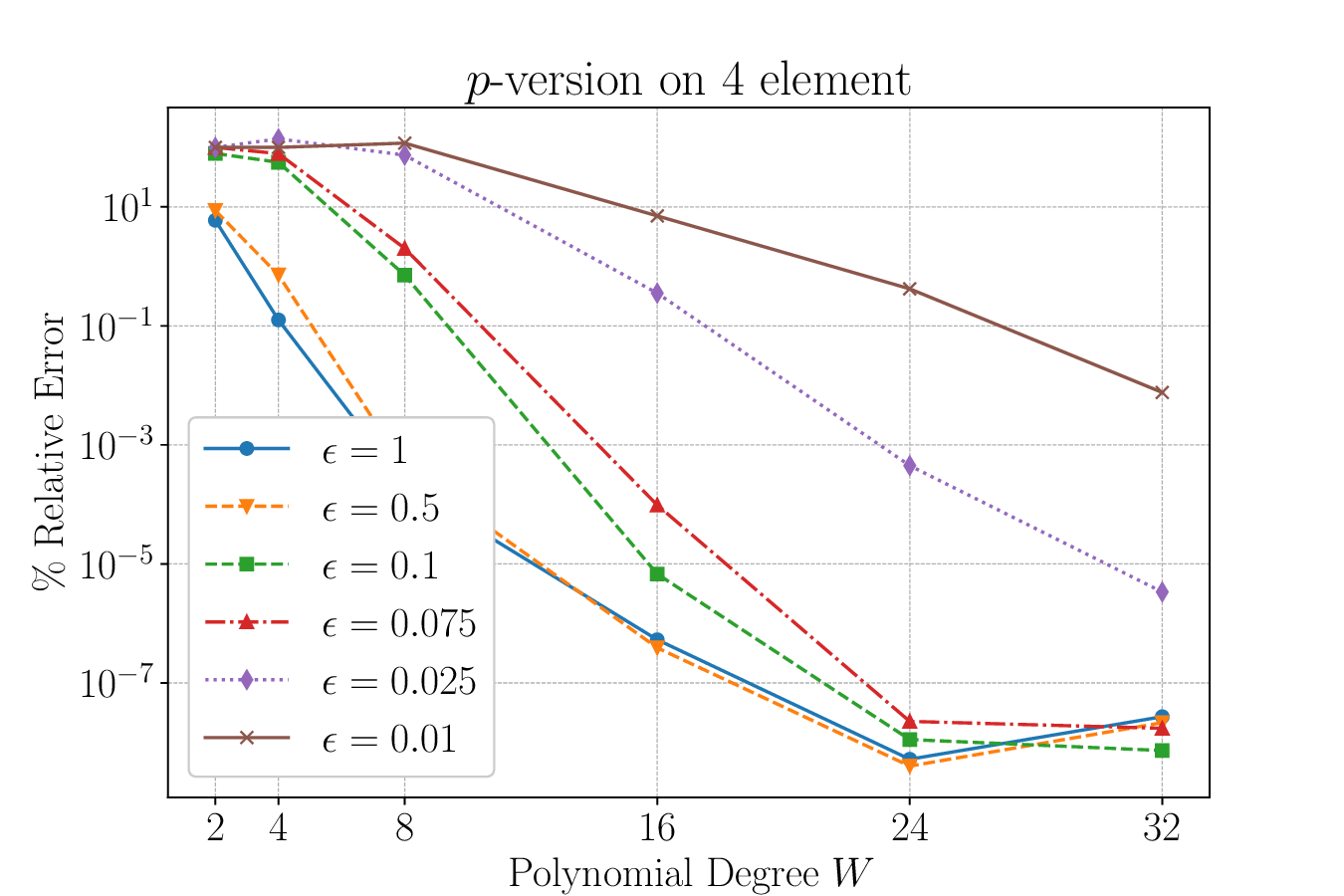}
\label{fig:ex2-c}
}
\hspace{-1.3cm}
\subfigure[]{
\includegraphics[width=0.50\textwidth]{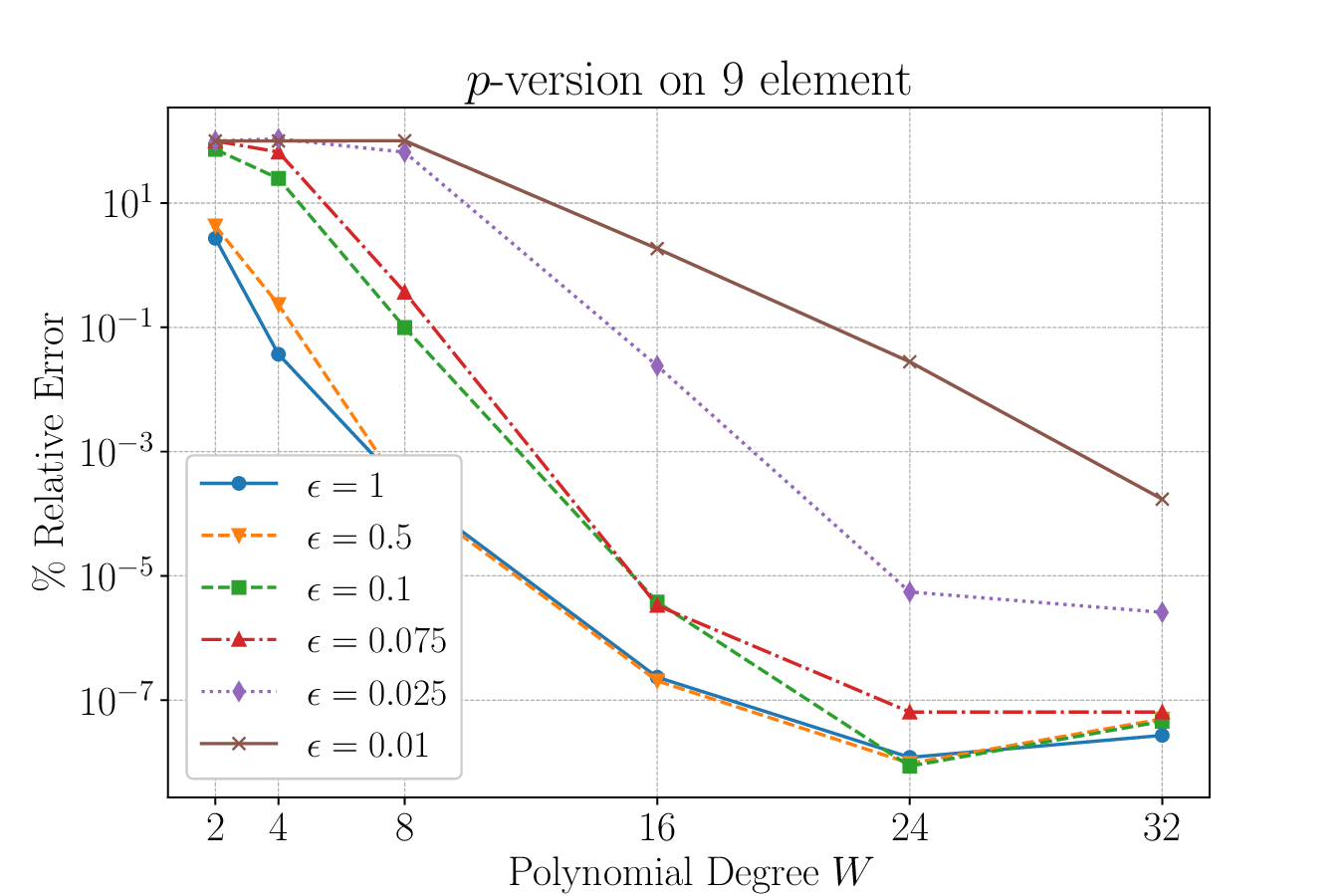}
\label{fig:ex2-d}
} 
\caption{Performance of the $p$-version on different mesh configurations (Example~\ref{example2}).}
\label{fig:ex2}
\end{figure}

\newpage
\begin{example}[\textbf{Boundary layer only on the right edge}] \label{example3}
\end{example}

Consider the boundary layer problem $-\epsilon^2\Delta w=f$, on the unit square domain $\Omega = (0,1)^2$ with $w=0$ on $\partial\Omega=\Gamma$. We choose the source term $f(x,y)$ such that the exact solution is
\begin{align*}
w_\epsilon(x,y) = x^2(1 - x)^2 y^2(1 - y)^2 e^{-(1 - x)/\epsilon},
\end{align*}
which exhibits steep gradients near the right edge ($x = 1$), resulting in the formation of a boundary layer (see Figure~\ref{Soln_Ex3}).

\begin{figure}[h!]
\centering
\includegraphics[width=1.0\textwidth]{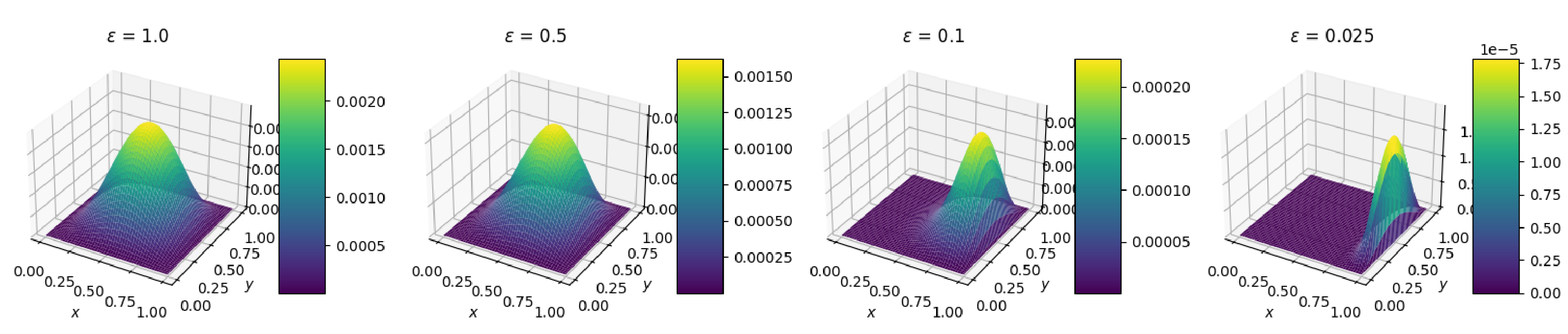}
\caption{Exact solution (Example~\ref{example3}).}
\label{Soln_Ex3}
\end{figure}

To resolve the boundary layer  simiar to Example~\ref{examp1}--\ref{example2} , we apply the $p$-version of the spectral element method on uniform meshes for various values of $\epsilon$. The corresponding convergence behavior with respect to the polynomial degree $W$ on semi $\log$-scale is shown in Figure~\ref{fig:ex3}, and the detailed relative error in $H_{\epsilon}^{2}$-norm are reported in Tables~\ref{ex3_ele1}--\ref{ex3_ele9}.

\begin{table}[h!]
\centering
\caption{Relative error for different scales of $\epsilon$ for Example~\ref{example3} (on a single element)}
\label{ex3_ele1}
\begin{tabular}{c | c c c c c c}
\toprule
\textbf{ $W$} & \boldmath{$\epsilon = 1$} & \boldmath{$\epsilon = 0.5$} & \boldmath{$\epsilon = 0.1$} & \boldmath{$\epsilon = 0.075$} & \boldmath{$\epsilon = 0.025$} & \boldmath{$\epsilon = 0.01$} \\
\midrule
2  & $6.63\mathrm{E}{+}01$ & $9.57\mathrm{E}{+}01$ & $1.00\mathrm{E}{+}02$ & $1.00\mathrm{E}{+}02$ & $1.05\mathrm{E}{+}03$ & $2.67\mathrm{E}{+}14$ \\
4  & $2.22\mathrm{E}{+}01$ & $3.67\mathrm{E}{+}01$ & $1.00\mathrm{E}{+}02$ & $1.00\mathrm{E}{+}02$ & $1.00\mathrm{E}{+}02$ & $1.10\mathrm{E}{+}02$ \\
8  & $2.52\mathrm{E}{-}02$ & $1.01\mathrm{E}{-}01$ & $4.82\mathrm{E}{+}01$ & $1.00\mathrm{E}{+}02$ & $1.00\mathrm{E}{+}02$ & $1.00\mathrm{E}{+}02$ \\
16 & $1.43\mathrm{E}{-}05$ & $1.09\mathrm{E}{-}04$ & $2.47\mathrm{E}{-}02$ & $1.25\mathrm{E}{+}01$ & $3.75\mathrm{E}{+}01$ & $1.00\mathrm{E}{+}02$ \\
24 & $3.00\mathrm{E}{-}09$ & $3.31\mathrm{E}{-}09$ & $4.66\mathrm{E}{-}06$ & $1.24\mathrm{E}{-}05$ & $4.66\mathrm{E}{+}00$ & $4.49\mathrm{E}{+}01$ \\
32 & $1.34\mathrm{E}{-}10$ & $9.37\mathrm{E}{-}11$ & $5.39\mathrm{E}{-}10$ & $1.98\mathrm{E}{-}09$ & $6.44\mathrm{E}{-}05$ & $1.87\mathrm{E}{+}00$ \\
\bottomrule
\end{tabular}
\end{table}

\begin{table}[h!]
\centering
\caption{Relative error for different scales of $\epsilon$ for Example~\ref{example3} (on 2 elements)}
\label{ex3_ele2}
\begin{tabular}{c | c c c c c c}
\toprule
\textbf{ $W$} & \boldmath{$\epsilon = 1$} & \boldmath{$\epsilon = 0.5$} & \boldmath{$\epsilon = 0.1$} & \boldmath{$\epsilon = 0.075$} & \boldmath{$\epsilon = 0.025$} & \boldmath{$\epsilon = 0.01$} \\
\midrule
2  & $7.65\mathrm{E}{+}01$ & $9.69\mathrm{E}{+}01$ & $1.00\mathrm{E}{+}02$ & $1.00\mathrm{E}{+}02$ & $1.00\mathrm{E}{+}02$ & $2.60\mathrm{E}{+}03$ \\
4  & $2.66\mathrm{E}{+}00$ & $8.51\mathrm{E}{+}00$ & $1.00\mathrm{E}{+}02$ & $1.00\mathrm{E}{+}02$ & $1.00\mathrm{E}{+}02$ & $1.00\mathrm{E}{+}02$ \\
8  & $9.60\mathrm{E}{-}03$ & $6.57\mathrm{E}{-}02$ & $6.24\mathrm{E}{+}01$ & $6.20\mathrm{E}{+}01$ & $1.00\mathrm{E}{+}02$ & $1.00\mathrm{E}{+}02$ \\
16 & $1.19\mathrm{E}{-}05$ & $9.97\mathrm{E}{-}05$ & $6.12\mathrm{E}{-}02$ & $5.87\mathrm{E}{-}01$ & $4.78\mathrm{E}{+}01$ & $1.00\mathrm{E}{+}02$ \\
24 & $3.32\mathrm{E}{-}09$ & $4.13\mathrm{E}{-}09$ & $7.57\mathrm{E}{-}06$ & $1.50\mathrm{E}{-}05$ & $1.13\mathrm{E}{-}02$ & $9.34\mathrm{E}{+}00$ \\
32 & $1.33\mathrm{E}{-}08$ & $1.14\mathrm{E}{-}08$ & $2.92\mathrm{E}{-}08$ & $5.67\mathrm{E}{-}08$ & $4.58\mathrm{E}{-}06$ & $1.40\mathrm{E}{-}03$ \\
\bottomrule
\end{tabular}
\end{table}
\begin{table}[h!]
\centering
\caption{Relative error for different scales of $\epsilon$ for Example~\ref{example3} (on 4 elements)}
\label{ex3_ele4}
\begin{tabular}{c | c c c c c c}
\toprule
\textbf{ $W$} & \boldmath{$\epsilon = 1$} & \boldmath{$\epsilon = 0.5$} & \boldmath{$\epsilon = 0.1$} & \boldmath{$\epsilon = 0.075$} & \boldmath{$\epsilon = 0.025$} & \boldmath{$\epsilon = 0.01$} \\
\midrule
2  & $8.13\mathrm{E}{+}01$ & $9.92\mathrm{E}{+}01$ & $1.00\mathrm{E}{+}02$ & $1.00\mathrm{E}{+}02$ & $1.00\mathrm{E}{+}02$ & $1.66\mathrm{E}{+}03$ \\
4  & $2.44\mathrm{E}{+}00$ & $7.43\mathrm{E}{+}00$ & $1.00\mathrm{E}{+}02$ & $1.00\mathrm{E}{+}02$ & $1.00\mathrm{E}{+}02$ & $1.00\mathrm{E}{+}02$ \\
8  & $8.35\mathrm{E}{-}03$ & $2.22\mathrm{E}{-}02$ & $6.06\mathrm{E}{+}01$ & $1.00\mathrm{E}{+}02$ & $1.00\mathrm{E}{+}02$ & $1.00\mathrm{E}{+}02$ \\
16 & $9.34\mathrm{E}{-}06$ & $4.16\mathrm{E}{-}05$ & $1.44\mathrm{E}{-}02$ & $1.11\mathrm{E}{-}01$ & $4.77\mathrm{E}{+}01$ & $1.00\mathrm{E}{+}02$ \\
24 & $4.09\mathrm{E}{-}09$ & $4.80\mathrm{E}{-}09$ & $6.94\mathrm{E}{-}07$ & $7.47\mathrm{E}{-}06$ & $2.93\mathrm{E}{-}03$ & $9.49\mathrm{E}{+}00$ \\
32 & $1.92\mathrm{E}{-}08$ & $1.67\mathrm{E}{-}08$ & $1.36\mathrm{E}{-}08$ & $1.91\mathrm{E}{-}07$ & $7.32\mathrm{E}{-}06$ & $1.40\mathrm{E}{-}03$ \\
\bottomrule
\end{tabular}
\end{table}
\begin{table}[h!]
\centering
\caption{Relative error for different scales of $\epsilon$ for Example~\ref{example3} (on 9 elements)}
\label{ex3_ele9}
\begin{tabular}{c | c c c c c c}
\toprule
\textbf{ $W$} & \boldmath{$\epsilon = 1$} & \boldmath{$\epsilon = 0.5$} & \boldmath{$\epsilon = 0.1$} & \boldmath{$\epsilon = 0.075$} & \boldmath{$\epsilon = 0.025$} & \boldmath{$\epsilon = 0.01$} \\
\midrule
2  & $2.89\mathrm{E}{+}01$ & $9.90\mathrm{E}{+}01$ & $1.00\mathrm{E}{+}02$ & $1.00\mathrm{E}{+}02$ & $1.00\mathrm{E}{+}02$ & $1.00\mathrm{E}{+}02$ \\
4  & $9.57\mathrm{E}{-}01$ & $5.43\mathrm{E}{+}00$ & $1.00\mathrm{E}{+}02$ & $1.00\mathrm{E}{+}02$ & $1.00\mathrm{E}{+}02$ & $1.00\mathrm{E}{+}02$ \\
8  & $8.48\mathrm{E}{-}03$ & $2.74\mathrm{E}{-}02$ & $7.26\mathrm{E}{+}01$ & $1.00\mathrm{E}{+}02$ & $1.00\mathrm{E}{+}02$ & $1.00\mathrm{E}{+}02$ \\
16 & $1.38\mathrm{E}{-}05$ & $3.72\mathrm{E}{-}05$ & $1.98\mathrm{E}{-}02$ & $5.18\mathrm{E}{-}01$ & $4.92\mathrm{E}{+}01$ & $1.00\mathrm{E}{+}02$ \\
24 & $7.22\mathrm{E}{-}09$ & $7.74\mathrm{E}{-}09$ & $6.94\mathrm{E}{-}07$ & $3.75\mathrm{E}{-}06$ & $4.99\mathrm{E}{-}03$ & $1.00\mathrm{E}{+}01$ \\
32 & $3.70\mathrm{E}{-}08$ & $3.35\mathrm{E}{-}08$ & $1.36\mathrm{E}{-}08$ & $1.07\mathrm{E}{-}07$ & $6.02\mathrm{E}{-}06$  & $2.06\mathrm{E}{-}04$ \\
\bottomrule 
\end{tabular}
\end{table}

On a single-element mesh (Figure~\ref{fig:ex3}(a)), exponential convergence is observed for large $\epsilon$, but the method fails to resolve the boundary layer for small $ \epsilon$, with error saturation or deterioration at higher $W$. As the number of elements increases to 2 and 4 (Figures~\ref{fig:ex3}(b)--(c)), the method captures moderate boundary layers more effectively, although the accuracy remains limited for small $\epsilon$. On a 9-element mesh (Figure~\ref{fig:ex3}(d)), substantial improvement is achieved across all values of $\epsilon$, with robust exponential convergence clearly visible for $\epsilon \geq 0.075$, and noticeable error reduction even for $\epsilon = 0.01$.

The convergence behavior shown in Figure~\ref{fig:ex3} demonstrates that while the $p$-version effectively resolves smooth solutions, resolving sharp boundary layers requires mesh refinement. For small values of $\epsilon$,  further improvement would require layer-adapted methods.

\begin{figure}[h!]
\centering
\centering
\subfigure[]{
\includegraphics[width=0.50\textwidth]{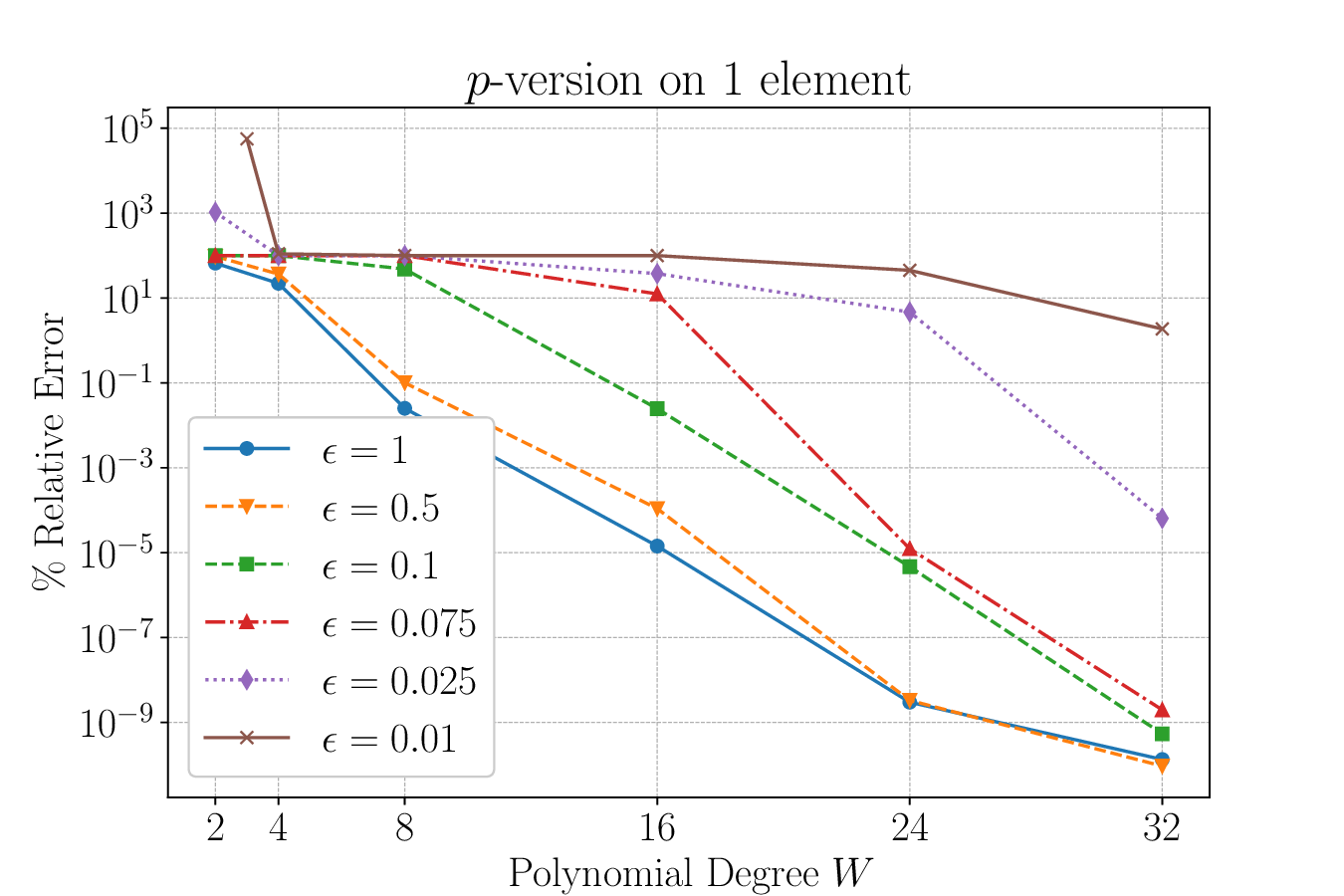}
\label{fig:ex3-a}
}
\hspace{-1.3cm}
\subfigure[]{
\includegraphics[width=0.50\textwidth]{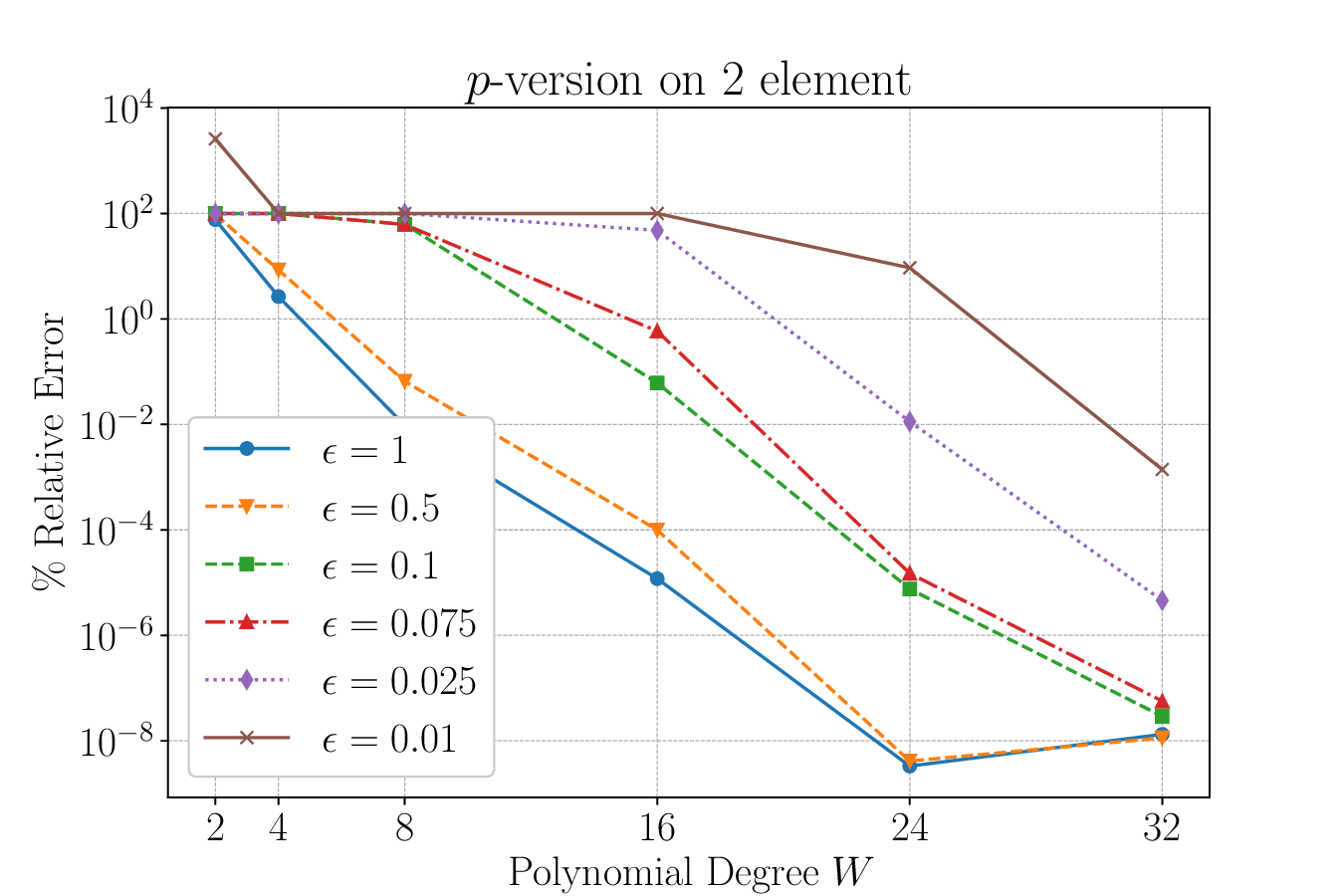}
\label{fig:ex3-b}
}
\hspace{0.5cm}
\subfigure[]{
\includegraphics[width=0.50\textwidth]{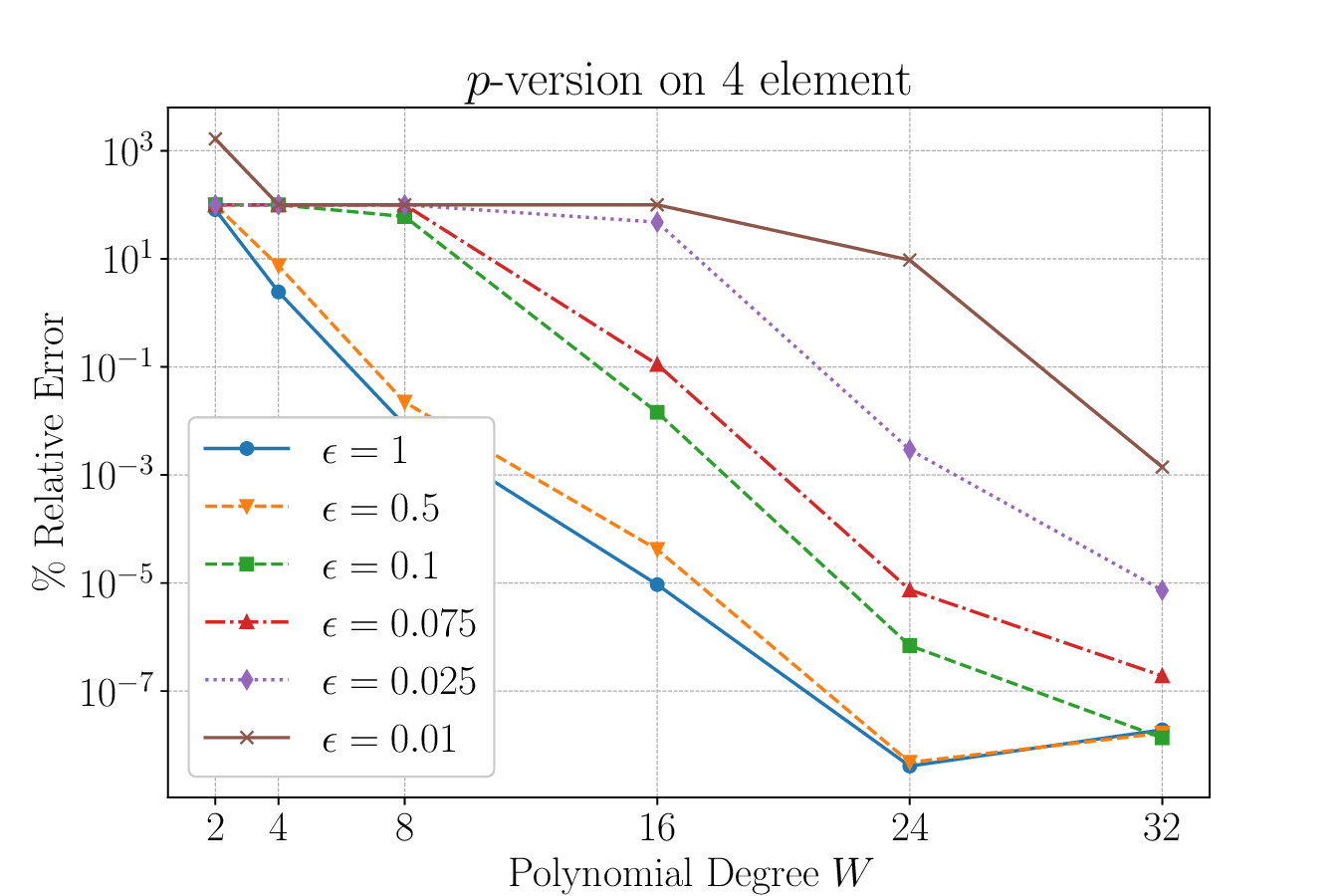}
\label{fig:ex3-c}
}
\hspace{-1.3cm}
\subfigure[]{
\includegraphics[width=0.50\textwidth]{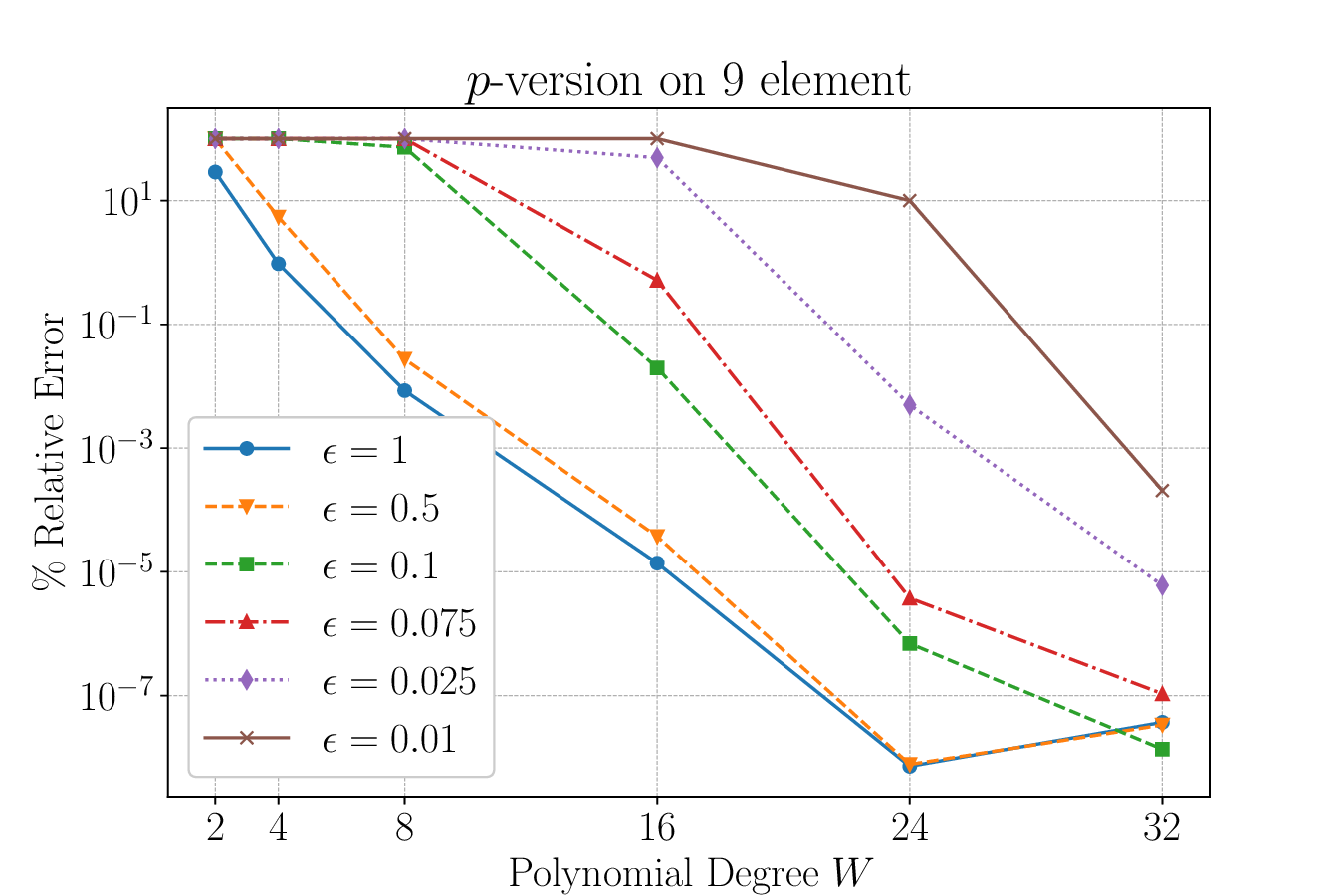}
\label{fig:ex3-d}
}
\caption{Performance of the $p$-version on different mesh configurations (Example~\ref{example3}).}
\label{fig:ex3}
\end{figure}

We now examine our method for the $rp$-version on a rectangular domain, employing a spectral boundary layer mesh (as shown in Figure~\ref{fig3}) by suitably chosen test problems.

\begin{example}[\textbf{Boundary layer only on the left edge}]\label{example4}
\end{example}
Consider the boundary layer problem~(\ref{eq2.6})--(\ref{eq2.7}) on $\Omega = (0,1)^2$ with $a = 0$, where the forcing function $f(x,y)$ is chosen so that the exact solution is
\begin{align*}
w_\epsilon(x,y) = x^2(1 - x)^2 y^2(1 - y)^2 e^{-x/\epsilon},
\end{align*}
where the solution develops a sharp boundary layer near the left edge ($x = 0$). Figure~\ref{Soln_Ex4} shows the exact solution highlighting the steep layer near the boundary.

\begin{figure}[h!]
\centering
\includegraphics[width=1.0\textwidth]{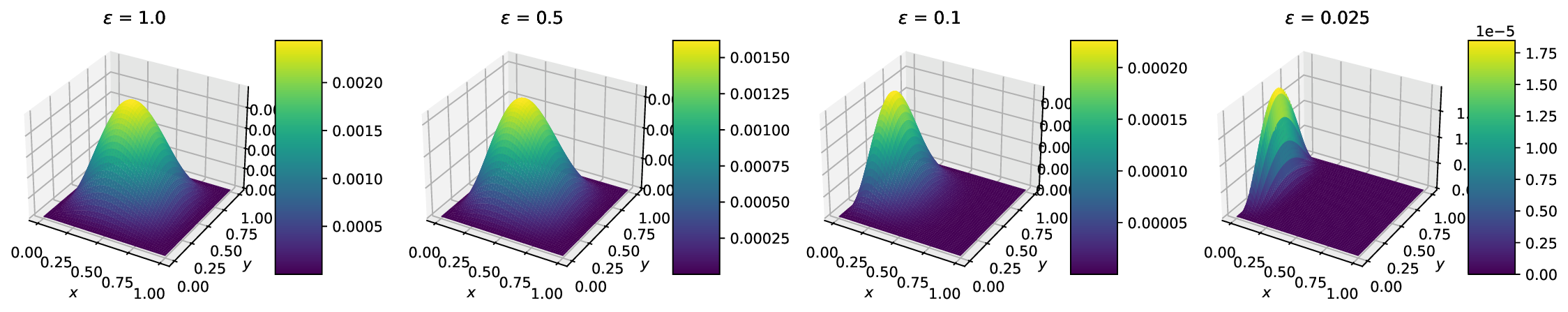}
\caption{Exact solution (Example 7.4).}
\label{Soln_Ex4}
\end{figure}

To resolve the boundary layer at $x = 0$, we used a spectral boundary layer mesh with two elements: a thin element near the layer and a coarse element covering the remainder of the domain, as similar to Figure~\ref{fig3}. The polynomial degree $W$ was varied independently on each element for accurate approximation.
As predicted by Theorem~\ref{thm6.2}, Table~\ref{ex4_blm} demonstrates robust exponential convergence in $H_{\epsilon}^{2}$-norm on semi $\log$-scale with respect to the polynomial degree $W$. For large $\epsilon$ values, the method quickly achieves machine-level accuracy. More importantly, even for small values such as $\epsilon = 0.01$, the error decreases rapidly once $W \geq 20$, confirming that the method accurately captures the boundary layer and regular (smooth) component of the solution (as shown in Figure~\ref{fig:ex4}). This validates the theoretical prediction that $rp$-version can deliver high accuracy with minimal elements when properly aligned to the layer structure.

\begin{table}[h!]
\centering
\caption{Relative error for different values of $\epsilon$ using spectral boundary layer mesh (Example~7.4)}
\label{ex4_blm}
\begin{tabular}{c | c c c c c c}
\toprule
\textbf{ $W$} & \boldmath{$\epsilon = 1$} & \boldmath{$\epsilon = 0.5$} & \boldmath{$\epsilon = 0.1$} & \boldmath{$\epsilon = 0.075$} & \boldmath{$\epsilon = 0.025$} & \boldmath{$\epsilon = 0.01$} \\
\midrule
2  & $7.58\mathrm{E}{+}01$ & $5.83\mathrm{E}{+}01$ & $1.00\mathrm{E}{+}02$ & $1.00\mathrm{E}{+}02$ & $1.00\mathrm{E}{+}02$ & $1.00\mathrm{E}{+}02$ \\
4  & $2.45\mathrm{E}{+}00$ & $6.40\mathrm{E}{+}00$ & $7.35\mathrm{E}{+}01$ & $8.80\mathrm{E}{+}01$ & $1.00\mathrm{E}{+}02$ & $1.00\mathrm{E}{+}02$ \\
8  & $2.97\mathrm{E}{-}05$ & $9.95\mathrm{E}{-}04$ & $2.34\mathrm{E}{+}00$ & $4.63\mathrm{E}{+}00$ & $6.13\mathrm{E}{+}01$ & $7.99\mathrm{E}{+}01$ \\
16 & $1.52\mathrm{E}{-}08$ & $3.64\mathrm{E}{-}08$ & $7.13\mathrm{E}{-}05$ & $2.03\mathrm{E}{-}04$ & $1.34\mathrm{E}{+}01$ & $3.33\mathrm{E}{+}01$ \\
24 & $2.42\mathrm{E}{-}09$ & $2.10\mathrm{E}{-}09$ & $5.00\mathrm{E}{-}08$ & $2.28\mathrm{E}{-}07$ & $8.06\mathrm{E}{-}05$ & $2.64\mathrm{E}{-}01$ \\
32 & $1.33\mathrm{E}{-}08$ & $4.48\mathrm{E}{-}08$ & $6.62\mathrm{E}{-}09$ & $2.63\mathrm{E}{-}08$ & $7.39\mathrm{E}{-}06$ & $5.08\mathrm{E}{-}04$ \\
\bottomrule
\end{tabular}
\end{table}

\begin{figure}[h!]
\centering
{\includegraphics[width=0.50\textwidth]{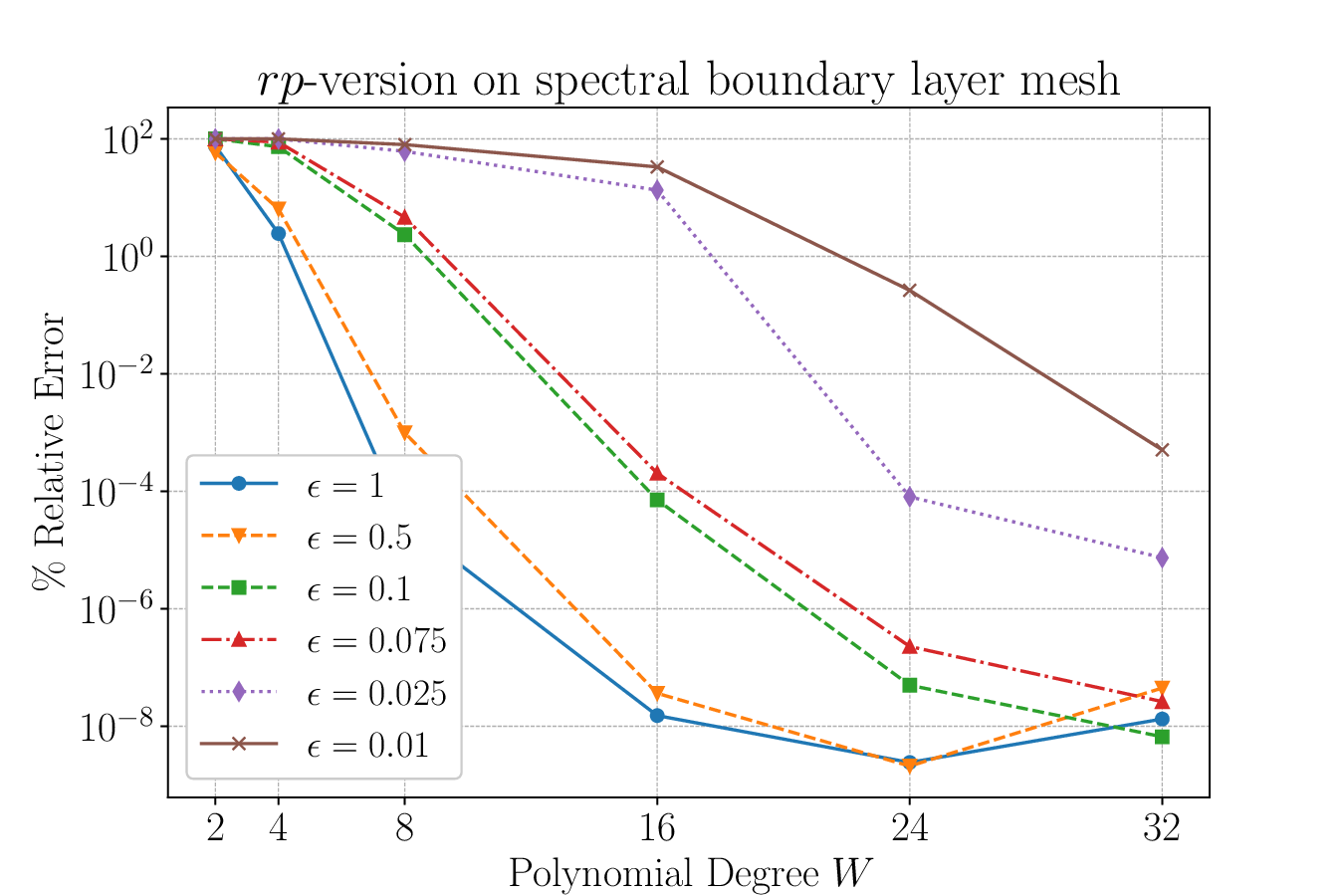}}
\caption{Performance of the $rp$-version on spectral boundary layer mesh (Example~\ref{example4})}
\label{fig:ex4}
\end{figure}

\newpage

\begin{example}[\textbf{Boundary layer only on the right edge}]\label{example5}
\end{example}

We now consider the boundary layer problem 

\begin{align*}
-\epsilon^2\Delta w+w=f\quad\mbox{in}\quad\Omega=(-1,1)^2,
\end{align*}
\begin{align*}
w=0\quad\mbox{on}\quad\partial\Omega=\Gamma.
\end{align*}

where $f$ is chosen such that 
\begin{align*}
w_\epsilon(x,y) = \sin(\pi y)\left(1 - e^{-(1-x)/\epsilon}\right),
\end{align*}

is the exact solution.This function exhibit sharp boundary layer along the right edge ($x = 1$),as shown in Figure~\ref{Soln_Ex5}.

\begin{figure}[h!]
\centering
\includegraphics[width=1.0\textwidth]{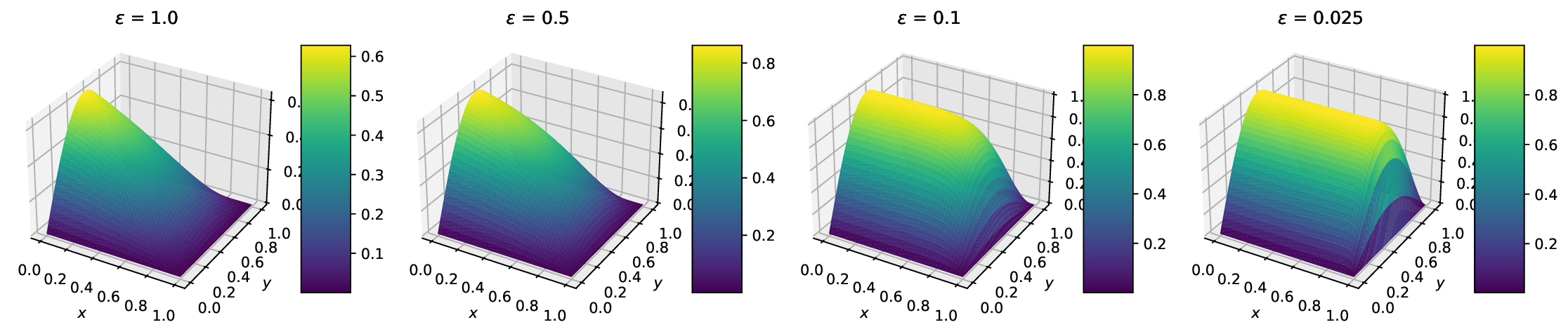}
\caption{Exact solution (Example 7.5).}
\label{Soln_Ex5}
\end{figure}

To resolve this layer, we employ a boundary layer mesh simiar to Example~\ref{example4}, near the right edge. The relative errors for different values of $\epsilon$ against $W$ on a semi $\log$-scale are reported in Table~\ref{ex5_blm}. For $\epsilon = 1$ and $\epsilon = 0.5$, rapid exponential convergence is observed, with the error dropping below $10^{-7}$ by polynomial degree $W = 16$.

As $\epsilon$ decreases, the solution have steep gradients near the right edge, so the initial convergence slows slightly. For example, at $\epsilon = 0.1$, the error remains around $10^{-3}$ for $W = 12$, but then decreases rapidly with increasing $W$, reaching $10^{-8}$ at $W = 28$. This behavior becomes even more pronounced for smaller values of $\epsilon$, such as $\epsilon = 0.01$, where the boundary layer is extremely sharp. Nevertheless, the method still achieves substantial error reduction (e.g., from $1.2 \times 10^1$ at $W=4$ to $5.56 \times 10^{-2}$ at $W=32$), which indicates that even when sharp boundary layers presents, exponential convergence is achieved at high polynomial degrees as seen in the convergence behaviour in Figure~\ref{fig:ex5}.

These results confirm the robustness and efficiency of the $rp$-version on boundary layer meshes and validate the theoretical estimate in Theorem~\ref{thm6.2}.

\begin{table}[h!]
\centering
\caption{Relative error for different values of $\epsilon$ using spectral boundary layer mesh (Example~7.5)}
\label{ex5_blm}
\begin{tabular}{c | c c c c c c}
\toprule
\textbf{ $W$ } & \boldmath{$\epsilon = 1$} & \boldmath{$\epsilon = 0.5$} & \boldmath{$\epsilon = 0.1$} & \boldmath{$\epsilon = 0.075$} & \boldmath{$\epsilon = 0.025$} & \boldmath{$\epsilon = 0.01$} \\
\midrule
2  & $9.24\mathrm{E}{+}17$ & $6.41\mathrm{E}{+}16$ & $4.42\mathrm{E}{+}16$ & $4.12\mathrm{E}{+}16$ & $3.38\mathrm{E}{+}16$ & $3.02\mathrm{E}{+}16$ \\

3  & $3.83\mathrm{E}{+}01$ & $2.58\mathrm{E}{+}01$ & $1.07\mathrm{E}{+}01$ & $1.06\mathrm{E}{+}01$ & $1.19\mathrm{E}{+}01$ & $1.27\mathrm{E}{+}01$ \\

4  & $2.89\mathrm{E}{+}01$ & $1.58\mathrm{E}{+}01$ & $7.49\mathrm{E}{+}00$ & $8.41\mathrm{E}{+}00$ & $1.10\mathrm{E}{+}01$ & $1.21\mathrm{E}{+}01$ \\
8  & $6.61\mathrm{E}{-}01$ & $4.96\mathrm{E}{-}01$ & $4.18\mathrm{E}{-}01$ & $1.06\mathrm{E}{+}00$ & $2.83\mathrm{E}{+}00$ & $3.33\mathrm{E}{+}00$ \\
16 & $1.71\mathrm{E}{-}07$ & $1.35\mathrm{E}{-}07$ & $8.10\mathrm{E}{-}07$ & $4.98\mathrm{E}{-}05$ & $4.43\mathrm{E}{-}01$ & $1.35\mathrm{E}{+}00$ \\
24 & $2.29\mathrm{E}{-}09$ & $1.85\mathrm{E}{-}09$ & $7.65\mathrm{E}{-}09$ & $1.12\mathrm{E}{-}08$ & $4.41\mathrm{E}{-}03$ & $4.77\mathrm{E}{-}01$ \\
32 & $1.01\mathrm{E}{-}08$ & $8.38\mathrm{E}{-}09$ & $2.02\mathrm{E}{-}08$ & $3.34\mathrm{E}{-}08$ & $2.61\mathrm{E}{-}06$ & $5.56\mathrm{E}{-}02$ \\
\bottomrule
\end{tabular}
\end{table}

\begin{figure}[h!]
\centering
{\includegraphics[width=0.50\textwidth]{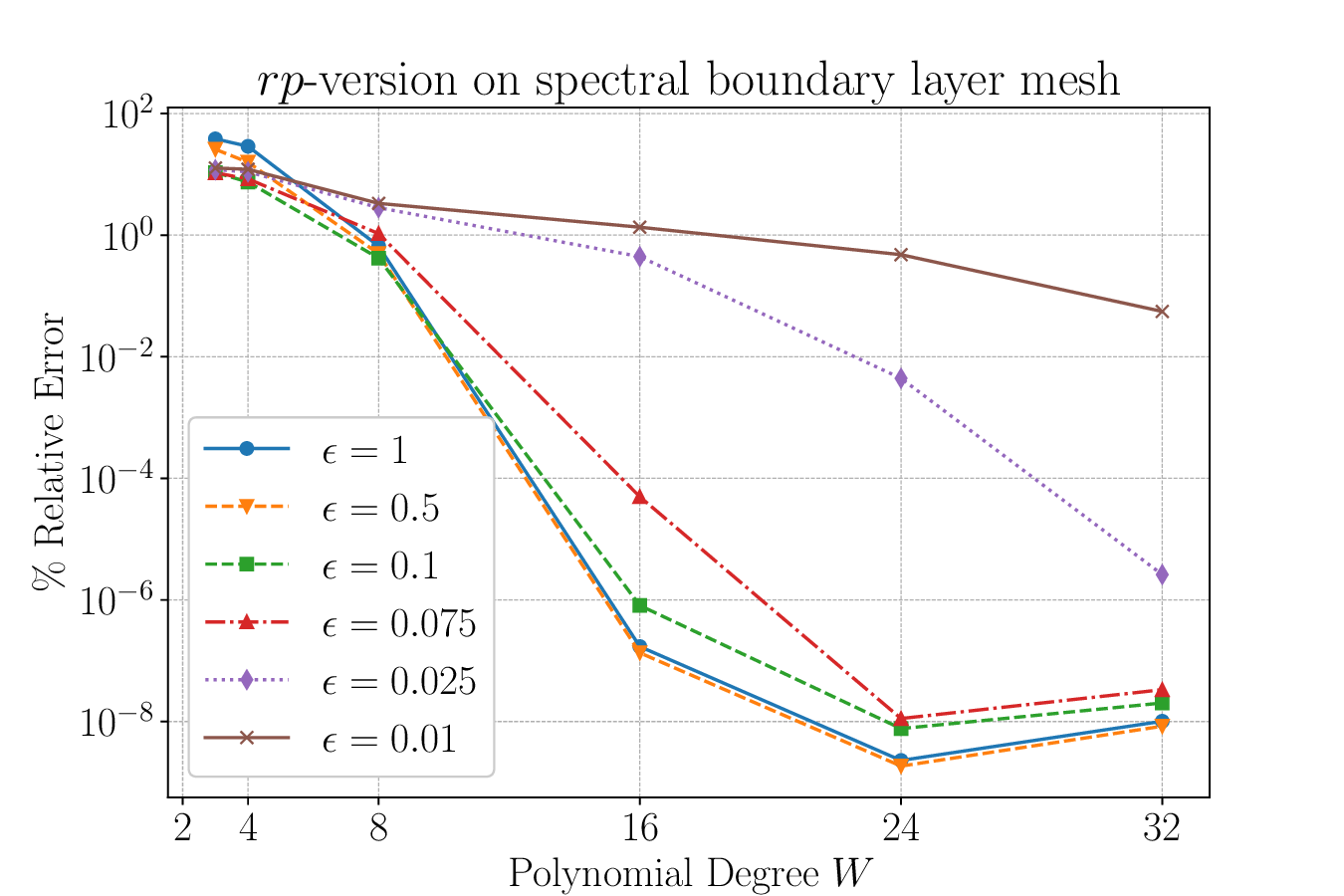}}
\caption{Performance of the $rp$-version on spectral boundary layer mesh (Example~\ref{example5})}
\label{fig:ex5}
\end{figure}

\section{Conclusion and future prospects}\label{sec8}

In this paper, we present a completely non-conforming formulation of the least-squares spectral element method for solving elliptic boundary layer problems in two dimensions on rectangular domains. To resolve boundary layers, we have used two types of mesh refinements, a $p$-version on a
fixed mesh and an $rp$-version on a layer adapted boundary layer mesh which consists of thin elements near the boundary layer and coarse
elements away from the layer. Stability estimates are derived using non-conforming spectral element functions. We have designed a numerical
scheme for both types of refinements which is based on minimising a least-squares functional in appropriate Sobolev norms. We construct robust preconditioners to control the condition number of the normal equations associated with the least-squares formulation involving parameter-weighted $H_{\epsilon}^2$-norm. The method is able to approximate boundary layers at a rate $O\left(\frac{\log W}{W^2}\right)$ for the $p$-version and at a rate $O({\epsilon}\alpha^{2W})$, uniformly in $\epsilon$, for the $rp$-version, where $0<\epsilon\leq1$ is the boundary layer parameter,
$\alpha<1$ is a constant and $W$ denotes the degree of the approximating polynomial. Numerical results are provided for model elliptic boundary layer problems at various scales of the boundary layer thickness, which confirm the efficiency and capability of
the method in resolving boundary layers on rectangular domains.

Numerical results validate both the theoretical error estimates and the predicted computational complexity of the method, which is also practically implementable on parallel computers due to its simple structure, involving matrix-vector multiplications at each iteration of the PCGM. In upcoming works, we plan to consider boundary layer problems on smooth domains with analytic boundary and to approximate boundary layer functions at a robust
exponentially rate independent of the boundary layer parameter.

\subsection*{Funding}
The authors received no financial support from any organization for the submitted work.
\subsection*{Data Availability Statement}
This article contains all data generated or analyzed during the current study.
\subsection*{Author Contributions}
Conceptualization, A.H., Sonia; methodology, A.H., Sonia; investigation, A.H., Sonia; validation, Sonia, A.K.; software, A.H. and Sonia; writing—original draft,
A.H. and Sonia; writing—review and editing, A.K., Sonia; supervision, A.K., A.H.; All authors have reviewed and approved the final version of the manuscript.
\subsection*{Declaration of AI use}
This article was prepared without the use of any AI-assisted technologies.
\subsection*{Conflict of Interest}
The authors have declared that there is no conflict of interest.

\appendix{}

\section{Technical Results}

\begin{LemmaA.1}
Let $\{\mathcal{F}_{u}\}\in S^{N,W}$. Then there exists $\{\mathcal{F}_{v}\}$ (where $\hat{v}_{l}=0$ on $\Gamma$ for all
$l=1,2,\cdots,N$) such that $v_{l}\in H^{2}(S)$ for $l=1,2,\cdots,N$ and $u+v\in H^{2}(\Omega)$. Moreover, there exists a
constant $C_\epsilon>0$ (depending upon $\epsilon$ and independent of $W$) and the estimate
\begin{align}\label{eqA.1}
\sum_{l=1}^{N}\|\hat{v}_{l}\|^{2}_{2,S}\leq C_\epsilon (\log W)^2 \left[\sum_{\Gamma_{l}^s\subset\Omega_{l}}\left(\|[u]\|^{2}_{0,\Gamma_{l}^s}
+\sum_{k=1}^{2}\|[u_{x_k}]\|^{2}_{\frac{1}{2},\Gamma_{l}^s}\right)\right],
\end{align}
holds.
\end{LemmaA.1}
\begin{proof}

Let $P_{j}, j=1,\cdots,4$, denote the vertices of $S$. A first correction $r_{l}(\xi,\eta)$ is applied at the vertices of
$S=(M_{l})^{-1}(\Omega_{l})$ so that
\begin{align}\label{eqA.2}
(u_{l}+r_{l})(P_j)&=\bar{u}(P_j)\;\;\mbox{and}\;\;((u_{l})_{x_k}+(r_{l})_{x_k})(P_j)=\bar{u}_{x_k}(P_j),
\end{align}
for $j=1,\cdots,4$ and $k=1,2$. Here $\bar{u}$ denotes the averages of $u^{W}$ at $P_j$ over all elements which have $P_{j}$ as a common vertex.

Now, we can always construct a polynomial $r_{l}(\xi,\eta)$ (as described in Section 3.5 of \cite{Tomar}) defined on $S$ such that
$r_{l}(P_j)=a_j, (r_{l})_{x_{1}}(P_j)=b_j$ and $(r_{l})_{x_{2}}(P_j)=c_j$ for $j=1,\cdots,4$. The values of $a_j, b_j$
and $c_j$ are calculated from equation (\ref{eqA.2}). Moreover, the polynomial $r_{l}(\xi,\eta)$ satisfies the estimate
\begin{align}\label{eqA.3}
\|r_{l}\|_{{2},S}^{2}\le C\sum_{j=1}^{4}\left(|a_{j}|^2+|b_{j}|^2+|c_{j}|^2\right).
\end{align}
By integrating equation (7.8) from~\cite{TW} and combining it with the analysis presented in Appendix C of~\cite{HUS1} yields

\begin{align}\label{eqA.4}
\sum_{l=1}^{N}||{r}_{l}||^{2}_{2,S}\leq C(\ln W)\left[\sum_{\Gamma^{s}_l\subset{\Omega_{l}}}\left(\|[u]\|^{2}_{0,\Gamma^{s}_l}+\sum_{k=1}^{2}\|[u_{x_k}]\|^{2}_{\frac{1}{2},\Gamma^{s}_l}\right)\right].
\end{align}
Define $y_{l}(\xi,\eta)=u_{l}(\xi,\eta)+r_{l}(\xi,\eta)$. Now we make a correction $\{s_{l}\}_{l}$ on the sides of
$S=(M_{l})^{-1}(\Omega_{l})$ such that if $t=(t_1,t_2)$ is a point on a side $I_j$ of $S=(M_{l})^{-1}(\Omega_{l})$, then
\begin{align}\label{eqA.5}
(y_{l}+s_{l})(t_1,t_2)&=\bar{y}(t_1,t_2),\;\;\mbox{and}\;\;((y_{l})_{x_k}+(s_{l})_{x_k})(t_1,t_2)=\bar{y}_{x_k}(t_1,t_2),\;\mbox{for}\;k=1,2.
\end{align}
Here, $\bar{y}(t_1,t_2)$ represents the average of $y_l$ evaluated at the point $(t_1, t_2)$, taken over all elements that share $I_j$ as one of their sides. Moreover, $s_{l}(t_1,t_2)=0$ if $(t_1,t_2)$ is a vertex of $S$. Furthermore, the polynomial $s_{l}(\xi,\eta)$
satisfies the estimate
\begin{align}\label{eqA.6}
||s_{l}||_{2,S}^{2}\le C\sum_{j=1}^{4}\left(\|s_{l}\|^2_{0,I_j}+\sum_{k=1}^{2}\|(s_{l})_{x_k}\|^2_{0,I_j}\right).
\end{align}
By Lemma 7.11 of~\cite{TW}, equation~(\ref{eqA.6}) gives
\begin{align}\label{eqA.7}
\sum_{l=1}^{N}\|{s}_{l}\|^{2}_{2,S}\leq C (\ln W)\left[\sum_{\Gamma^{s}_l\subset{\Omega_{l}}}\left(\|[u]\|^{2}_{0,\Gamma^{s}_l}+\sum_{k=1}^{2}\|[u_{x_k}]\|^{2}_{\frac{1}{2},\Gamma^{s}_l}\right)\right].
\end{align}
Define $\nu_{l}(\xi,\eta)=u_{l}(\xi,\eta)+r_{l}(\xi,\eta)+s_{l}(\xi,\eta)$. Now we make a correction $\{\delta_{l}\}_{l}$ on
the edges of the square $S=(M_{l})^{-1}(\Omega_{l})$ such that $\delta_{l}=0$ for all vertices of $S=(M_{l})^{-1}(\Omega_{l})$,
and $\delta_{l}\in H^{2}(S)$ for all $l$. Let
\begin{align}\label{eqA.8}
\delta_{l}|_{e_1}&=\mathcal{E}_1=\frac{1}{2}({\nu}_{l}-{\nu}_{m})|_{e_1},\nonumber\\
(\delta_{l})_{x_1}|_{e_1}&=\mathcal{F}_1=\frac{1}{2}({\nu}_{l}-{\nu}_{m})_{x_1}|_{e_1},\nonumber\\
(\delta_{l})_{x_2}|_{e_1}&=\mathcal{G}_1=\frac{1}{2}({\nu}_{l}-{\nu}_{m})_{x_2}|_{e_1}.
\end{align}
Here, $e_1$ denotes the edge common to $S=(M_{l})^{-1}(\Omega_{l})$ and $S=(M_{m})^{-1}(\Omega_{m})$. Similarly, we can define $\mathcal{E}_j$, $\mathcal{F}_j$, and $\mathcal{G}_j$ for edges with indices $j = 2, \ldots, 4$.

If ${e}_{j} \subseteq \Gamma \cup \Gamma_0$ for $j = 1, \ldots, 4$, we set $\mathcal{E}_j$, $\mathcal{F}_j$, and $\mathcal{G}_j$ to vanish identically. We now proceed to estimate the term

\begin{align*}
\int_{0}^{1}&\frac{ |(\partial_{\eta}\mathcal{E}_1)(1-t,-1)-\mathcal{G}_2(1,-1+t)|^2}{t} dt\nonumber\\
&\leq C\left\{\int_{0}^{1}\frac{ |(\partial_{\eta}\mathcal{E}_1)(-1+t,-1)|^2}{t} dt+ \int_{0}^{1}\frac{ |\mathcal{G}_2(-1,-1-t)|^2}{t} dt\right\}.
\end{align*}
Using Theorem 4.79 and Theorem 4.82 of~\cite{SCHW}, we have
\begin{align*}
\int_{0}^{1}&\frac{ |(\partial_{\eta}\mathcal{E}_1)(1-t,-1)-\mathcal{G}_2(1,-1+t)|^2}{t} dt\leq C (\ln W)^2(\|\partial_{\eta}\mathcal{E}_1\|_{1/2,\Gamma^{s}_l}+\|\mathcal{G}_2\|_{1/2,\Gamma_{l}^s}).
\end{align*}
In a similar manner, all terms presented in Theorem 5.2 of~\cite{BDM} can be estimated accordingly. Consequently, we obtain
\begin{align}\label{eqA.9}
\sum_{l=1}^{N}\|\delta_{l}\|^{2}_{2,S}\leq C (\log W)^2 \left[\sum_{\Gamma_{l}^s\subset{\Omega_{l}}}\left(||[u]||^{2}_{0,\Gamma_{l}^s}+\sum_{k=1}^{2}\|[u_{x_k}]\|^{2}_{\frac{1}{2},\Gamma_{l}^s}\right)\right].
\end{align}
Choosing $v_{l}(\xi,\eta)=r_{l}(\xi,\eta)+s_{l}(\xi,\eta)+\delta_{l}(\xi,\eta)$ and combining the results of equations (\ref{eqA.4}),
(\ref{eqA.7}) and (\ref{eqA.9}), the desired result follows.
\end{proof}
\begin{LemmaA.2}\label{lem2}
Let $w=u+v\in H^{2}(\Omega)$ and $\{\mathcal{F}_{u}\}\in S^{N,W}$ as described in Lemma A.1. Then the estimate
\begin{align}
\|w\|_{\frac{3}{2},\Gamma}^{2} \leq C(\log W)^2 & \left[\sum_{\Gamma_{l}^s\subseteq{\Omega}}^{2}\left(\|[u]\|^{2}_{0,\Gamma_{l}^s}+\sum_{k=1}^{2}\|[u_{x_k}]\|^{2}_{\frac{1}{2},\Gamma_{l}^s}\right)\right].
\end{align}
holds. Here, $\|v\|_{3/2,\Gamma_{l}^s}^{2}$ is defined as $\displaystyle{\inf_{q|_{\Gamma_{l}^s}=v}}\{\|q\|_{H^{2}(S)}\}$ as in~\cite{HUS1}.
\end{LemmaA.2}
\begin{proof}
From Lemma A.1 there exists $\{\mathcal{F}_{v}\}$ (where $\tilde{v}^{l}$, $l=1,2,\cdots,N$ is zero on $\Gamma$) for which $w=u+v\in H^{2}(\Omega)$.
Moreover, $w=u$ on $\Gamma$. By the Minkowski inequality, we obtain
\begin{align}\label{eqA.10}
\|w\|_{{\frac{3}{2}},\Gamma}^{2}&\leq C\sum_{\Gamma_{l}^s\subseteq \Gamma}\|u+v\|_{{\frac{3}{2}},\Gamma_{l}^s}^{2}\leq C\sum_{\Gamma_{l}^s\subseteq \Gamma}\left(\|u\|_{{\frac{3}{2}},\Gamma_{l}^s}^2+\|v\|_{{\frac{3}{2}},\Gamma_{l}^s}^2\right).
\end{align}
Using the result from Lemma A.1 in equation (\ref{eqA.10}), we get the result.
\end{proof}

\end{document}